\documentclass[11pt]{article}
\usepackage{latexsym}
\usepackage{amssymb}
\usepackage{amsfonts}
\usepackage{amsmath}
\usepackage{cancel}
\usepackage{xcolor}
\usepackage{hyperref}
\definecolor{bleu}{rgb}{0.00,0.4,0.90}
\definecolor{magenta}{rgb}{1.0, 0.0, 1.0}

\makeatletter
\long\def\unmarkedfootnote#1{{\long\def\@makefntext##1{##1}\footnotetext{#1}}}
\makeatother

\newcommand{\dx}{\,\mathrm d }

\newcommand{\reale}{{\mathbb R}}
\newcommand{\naturale}{{\mathbb N}}
\newcommand{\Rn}{\reale^N}

\newcommand{\Sob}[2]{W^{1,#1}_0\def\next{#2}\ifx\next\empty\else(#2)\fi}
\newcommand{\SobpO}{\Sob{\overrightarrow{p}}{\Omega}}
\newcommand{\Sobdual}[1]{W^{-1,p'}\def\next{#1}\ifx\next\empty\else(#1)\fi}

\newcommand{\nc}{\normalcolor}

\newtheorem{definition}{Definition}[section]

\newtheorem{lemma}[definition]{Lemma}

\newtheorem{theorem}[definition]{Theorem}

\newtheorem{proposition}[definition]{Proposition}

\newtheorem{corollary}[definition]{Corollary}

\newtheorem{remark}[definition]{Remark}

\newcommand{\Om}{\Omega}

\newcommand{\R}{\mathbb R}

\newcommand{\de}{\partial}
\newcommand{\RN}{{\mathbb R}^{N}}
\newcommand{\Si}{\sum_{i=1}^N}

\newcommand{\dei}{\partial_{x_i}}

\newcommand{\ee}{\varepsilon}

\newenvironment{proof}[1][Proof]{\textbf{#1.} }{\ \rule{0.5em}{0.5em}}
\begin{document}

\title{Existence and uniqueness of solutions to some anisotropic  elliptic equations with a singular convection term }

\author{ G. di
Blasio\thanks{Dipartimento di Matematica e Fisica,
Universit\`{a} degli Studi della Campania \textquotedblleft L. Vanvitelli\textquotedblright , Viale Lincoln, 5 - 81100
Caserta, Italy. E--mail: giuseppina.diblasio@unicampania.it} -- F.
Feo\thanks{Dipartimento di Ingegneria, Universit\`{a} degli Studi di
Napoli \textquotedblleft Pathenope\textquotedblright, Centro
Direzionale Isola C4 80143 Napoli, Italy. E--mail:
filomena.feo@uniparthenope.it}
-- G. Zecca \thanks{
Dipartimento di Matematica e Applicazioni \textquotedblleft R. Caccioppoli\textquotedblright,  Universit\`{a} degli Studi di Napoli Federico II,
Complesso Universitario  Monte S. Angelo, Via Cintia 26, Napoli, Italy. E--mail: g.zecca@unina.it }}

%\date{}

\maketitle
\begin{abstract}
We prove the existence and uniqueness of weak solutions to a class of anisotropic elliptic equations with coefficients of  convection term belonging to some suitable Marcinkiewicz spaces. Some useful a priori estimates and regularity results are also derived.
\end{abstract}

\bigskip

\footnotetext{\noindent\textit{Mathematics Subject Classifications: 35J25, 35B45}
\par
\noindent\textit{Key words: Noncoercive Anisotropic nonlinear equations, Lower order terms, Existence, Uniqueness, Regularity }}

\numberwithin{equation}{section}

\section{\textbf{Introduction}}

In this paper we obtain existence and uniqueness results for the weak solutions of  the following class of Dirichlet problems

\begin{equation}\label{P}
\left\{
\begin{array}[c]{ll}%
-\sum_{i=1}^N   \de_{x_i} \left [ \mathcal{A}_i(x,\nabla u)+\mathcal B_i(x,u)    \right]+ \mathcal G(x,u )
=    \mathcal{F}  &\hbox { in } \Omega,
\\
& \\
u=0  & \hbox{ on } \partial \Omega,
\end{array}
\right.
\end{equation}
where $\Omega $  is a bounded domain of $ \RN$ with Lipschitz boundary, $N>2$,  $p_i>1$ for every $i=1,...,N$ with
$ \bar p<N $, denoting by $\overline{p}$ the harmonic mean of $\vec{p}=
(p_1,\cdots,p_N)$, \textit{i.e.}
\begin{equation}\label{p bar}
\frac{1}{\overline{p}}=\frac{1}{N}\sum_{i=1}^N\frac{1}{p_i}.
\end{equation}

%%%%%%%%%%%%%%%%
\noindent Throughout this paper, we make the following assumptions for any $i=1,...,N,$
\bigskip

\noindent
$ (\mathcal H 1)\quad \mathcal A_i: \Om \times \mathbb R^N \rightarrow \R $ is a Carath\'eodory function
that satisfies
\begin{eqnarray}
\label{ii1}&&| \mathcal A_i (x, \xi ) | \leqslant \beta_i  | \xi_i|^{p_i-1}\\
\label{ii2}&&  \alpha  |\xi_i|^{p_i}\leqslant   \mathcal A_i(x, \xi )\, \xi_i\\
\label{ii3}&&  0  <    ( \mathcal A_i(x, \xi )- \mathcal A_i(x, \eta ) )\, (\xi_i -\eta_i ) \quad\,\,\xi\neq \eta
\end{eqnarray}

\noindent for a.e. $x \in \Om$ and for any vector $\xi,\eta$  in $\mathbb R^N$, where $0<\alpha\leqslant \beta_i$ are constants.
\bigskip

\noindent
$(\mathcal H 2) \quad \mathcal B_i:\Om \times \mathbb R \rightarrow \mathbb R $ is a Carath\'eodory  function such that
\begin{equation}\label{b}
|\mathcal B_i(x,s)| \leqslant b_i(x)|s|^{\frac{\bar p}{p_i'}},
\end{equation}
for a.e. $x\in \Om$ and for every $s\in \mathbb R$, with $b_i: \,\Om\to [0,+\infty)$ measurable function such that
\begin{equation}\label{bi}
 b_i\in L^{\frac{N\,p'_i} {\bar p},\infty}(\Om).
\end{equation}

\noindent $(\mathcal H 3) $  $ \mathcal G: \Om \times \mathbb R \rightarrow \mathbb R $ is a Carath\'eodory  function such that
\begin{equation}\label{b1}
|\mathcal G(x,s)| \leqslant \tilde{\mu} |s|^{\gamma}
\end{equation}
with $1\leq\gamma< p_\infty-1$ and
\begin{equation}\label{sign}
\mathcal G(x,s)\, s\geq0
\end{equation}
for a.e. $x\in \Om$ and for every $s\in \mathbb R$, where $\tilde{\mu}$ is a non negative  constant and  $p_\infty=\max\{\overline{p}^*,p_{\max}\}$,  with $p_{\max}= \max_i p_i$.

\medskip

\noindent
$(\mathcal H 4) $ $\mathcal{F}$ belongs to the dual space $ (W_0^{1,\vec p}(\Om))^*$, where $W_0^{1,\vec p}(\Om)$ is the anisotropic Sobolev space defined in Section \ref{preliminari}.

\medskip

We observe that in our assumptions, the following definition of weak solution is well posed.

\begin{definition}
For any  $\mathcal{F}\in (W_0^{1,\vec p}(\Om))^*$ we say that
 $u\in W_0^{1,\vec {p} }(\Omega )$ is a weak solution to  \eqref{P} provided
 \begin{equation}\label{sol1}
\int_\Om   \left[ \Si  \left[\mathcal{A}_i(x,\nabla u)+\mathcal B_i(x,u)    \right] \dei \varphi  + \mathcal G (x,u )  \varphi \right]  \, dx = \langle \mathcal{F}\, , \varphi \rangle
\end{equation}
$\forall \varphi\in C^\infty_0(\Om)$, where $\langle \cdot, \cdot \rangle$ denotes the duality product of $(W_0^{1,\vec p}(\Om))^*$ and $W_0^{1,\vec p}(\Om)$.
\end{definition}

\medskip

In the anisotropic framework  the natural space where looking for weak solutions of Dirichlet problem \eqref{P} is the anisotropic Sobolev space $W_0^{1,\overrightarrow{p}}(\Omega)$  (see $\S$ \ref{preliminari} for definition). When the harmonic mean $\bar p$, defined in \eqref{p bar}, is less then the dimension $N$, it is well-known (see $\S$ \ref{preliminari}) that $W_0^{1,\overrightarrow{p}}(\Omega)$ is continuously embedded in the Lorentz space $L^{\bar p^*,\bar p}(\Omega)$.  On the other hand, by Poincar\'{e} inequality (see \eqref{dis poincare} below) the space $W_0^{1,\overrightarrow{p}}(\Omega)$ is embedded in the Lebesgue space $L^{p_{\max}}(\Omega)$. This suggests us to link, as in our assumption  \eqref{b1}, the growths of zero order term with   $p_\infty$, which  takes over how the $p_i$ are spread out.
%hat depends on $p_\infty:=\max\{\bar p^*,p_{\max}\}$ as in our assumption \eqref{b}.

The prototype equation of  our class of problems \eqref{P} is
\begin{equation}\label{model 1}
-\sum_{i=1}^N \de_{x_i}  \left[|\de_{x_i}u|^{p_i-2}\de_{x_i}u+\beta_i(x)|u|^{ \frac{\bar{p} }{p_i'}-1}u\right] +\tilde{\mu} |u|^{\gamma-1}u =    \mathcal{F}  \qquad \mbox { in } \Omega,
\end{equation}
where $p_i\geq1, \bar p<N,\tilde{\mu} \geq0,1\leq\gamma<p_\infty -1,\beta_i\in L^{\frac{N\,p'_i}{\bar p},\infty}(\Om)$ and $\mathcal{F}$ belongs to the dual space. We stress that when $\bar p^*\geq p_{\max}$ using the Sobolev  embedding in the Lorentz space $L^{\bar p^*,\bar p}(\Omega)$  the  summability assumption on $\beta_i$ is optimal to assure $|\beta_i|^{p'_i}|u|^{\bar p}\in L^1(\Omega)$. Moreover we observe that when $p_i=p$, $\beta_i=0$ and $\widetilde{\mu}=0$, the principal part in  equation \eqref{model 1} becomes the so-called pseudo-Laplacian operator (see \cite{L} pp. 106 and 155) or orthotropic $p$-Laplacian operator (see \cite{BB}), extensively studied in the literature.

Let us point out that term anisotropy is used in various scientific disciplines and could have a different meaning when it is related to equations as well. The interest in anisotropic problems has deeply increased in the last years since their many applications in mathematical modelling for natural phenomena in biology and fluid mechanics. For example, they are related to the mathematical description of the dynamics
of fluids in anisotropic media when the conductivities of the media are different in different directions (see e.g. \cite{ADS}) and they also appear in biology as a model for the propagation of epidemic diseases in heterogeneous domains
(see e.g. \cite{BK}). On anisotropic problems many results in different directions have been obtained, here we quote a list of references that is obviously not exhaustive and we refer the reader to references therein to extend it: \cite{ACC,AdBF1,AC,BGM,BB,BaCri,Barbara,C2,GGG,FVV,P,Dic1,Fragala,LI,Str}.

 A goal of this paper is to analyze the existence of weak solutions for this class of problems (see \cite{FGMZ} for the isotropic case). %We point out that  here we allow the $\mathcal G$-term to vanish.
  Since under our assumptions  the coercivity for the involved operator in problem \eqref{P} is not guaranteed, we will proceed as usual by approximations.  The existence of weak solutions can be expected when the datum and the coefficients are smooth enough. Then we first consider problem with $b_i\in L^\infty(\Omega)$ and then we reduce to the general case $b_i\in L^{\frac{N\,p'_i} {\bar p},\infty}(\Omega)$ assuming a control on suitable distance of $b_i$ from $L^\infty(\Omega)$ (see  assumption \eqref{hpdistmain} in Theorem \ref{TH E1}). This strategy allows us to overcome that the norms in  the Marcinkiewicz $L^{\frac{N\,p'_i} {\bar p},\infty}$ are not absolutely continuous and that $L^\infty(\Omega)$  is not dense in Marcinkiewicz spaces.
 Finally when we pass to the limit in the approximated problems  we have to deal with an extra difficulty due to the fact that in our assumption
   the operator $u\in W_0^{1,\overrightarrow{p}}(\Omega)\rightarrow {b_i}(x) |u|^{\frac{\bar p}{p'_i}-1}u\in L^{p'_i}(\Omega)$ is not compact in general. We emphasize that our assumption on the distance \eqref{hpdistmain}, firstly considered  in \cite{GGM} in the isotropic case (see also \cite{GMZ1}, \cite{FGMZ}), is weaker than asking  the smallness of the norms of $b_i$  which is the standard approach to treat the presence of the first order term.

We also analyze the regularity for weak solutions of problem \eqref{P} and we obtain Stampacchia type regularity, extending \cite{BMS, Str, Dic1} where the case $\mathcal B\equiv 0$ is studied. Regularity results for local solutions of problem \eqref{P} has been recently obtained in \cite{nostro}.

\medskip

For what concerns the uniqueness (see Theorems \ref{THU1}, \ref{THU2} and \ref{THU3}) we emphasize that our proofs strongly uses the presence of the zero order term and his monotonicity assumption \eqref{cond G} below. In order to describe our result, let us reduce
ourselves to the model case \eqref{model 1} for simplicity.  The main difficulty is due to the presence of the lower order terms that are only H\"{o}lder continuous, but not Lipschitz continuous with respect to  solution $u$ when $p_i<2$ for all $i$.
We recall  \cite{BP} for equations with a  H\"{o}lder continuous dependence on the solution of the coefficient in the main part of the operator.

If one wants to neglect the term $\mathcal G$, Corollary \ref{Cor unic} gives a partial uniqueness result when the datum is zero (see \cite{B2012}  for isotropic case).
Otherwise also in the isotropic case the uniqueness is proved in \cite{B2012} requiring  a control of the partial derivative of first order term $B(x,u)$ with respect to $u$, which is not satisfied in the simple case $B(x,u)=\textbf{b}(x)|u|^{p-2}u$ with  $\textbf{b}\in \left(L^{\frac{N}{p-1}}(\Omega)\right)^N$ .

\medskip

The paper is organized as follows. Section 2 contains some preliminaries and in Section 3 we prove some useful a  priori estimates. In our strategy we need some $L^\infty$ estimate of $u$ as well. For completeness we write Section 4 to analyze the regularity of the solutions. The existence theorem is stated and proved in Section 5. The last section is devoted to the uniqueness results.

\section{Preliminaries}\label{preliminari}

In the present section we recall some known function spaces, useful  in the sequel.

We start by recalling definitions of Lorentz spaces and their properties (see \cite{PKOS}  for more details).

Here we  assume that $\Om\subset \RN$, $N>2$ is an open set. Given $1<p,q<+\infty$, the Lorentz space $L^{p,q}(\Om)$ consists of all measurable functions $f$ defined on $\Omega$ for which the quantity
\begin{equation}\label{norma Lorentz}
\|f\|_{p,q}^q=p\int_{0}^{+\infty} |\Om_t|^{\frac{q}{p}}t^{q-1}d t
\end{equation}
is finite, where $\Om_t= \left\{ x\in \Om: |f(x)|>t \right\}$ and $|\Om_t|$ is the Lebesgue measure of $\Om_t$, that is, $\mu_f(t)=|\Om_t|$ is the distribution function of $f$. Note that $\|\cdot\|_{p,q}$  is equivalent to a norm and $L^{p,q}(\Omega)$ becomes a Banach space when endowed with it. For $p=q$, the Lorentz space
$L^{p,p}(\Omega)$ reduces to the Lebesgue space $L^p(\Omega)$. For $q=\infty$, the class $L^{p,\infty}(\Omega)$ consists of all measurable functions $f$ defined on $\Omega$ such that
\begin{equation}\label{2.1}
\|f\|^p_{p,\infty}=\sup_{t>0}t^{ p}\mu_f(t)<+\infty
\end{equation}
and it coincides with the Marcinkiewicz class or the so-called weak-$L^p(\Omega)$.

It well-known that if $\Omega$ is bounded, the following inclusions hold
\begin{equation}\label{LorLebEmb}
L^r (\Om)\subset L^{p,q}(\Om)\subset  L^{p,r} (\Om) \subset L^{p,\infty}(\Om)\subset L^q(\Om),
\end{equation}
whenever $1\leqslant q<p<r\leqslant \infty.$ Moreover, for $1<p<\infty$, $1\leqslant q\leqslant \infty$ and $\frac 1 p+\frac 1 {p'}=1$, $\frac 1 q+\frac 1 {q'}=1$, if $f\in L^{p,q}(\Omega)$, $g\in L^{p',q'}(\Omega)$ we have the H\"{o}lder--type inequality
\begin{equation*}\label{holder}
\int_{\Omega}|f(x)g(x)|dx \leqslant \|f\|_{p,q}\|g\|_{p',q'}.
\end{equation*}

\bigskip

 We remark that $L^\infty(\Om)$ is not dense in $L^{p,\infty}(\Om)$ for $p\in{}]1, +\infty[$. Therefore it is possible to define the distance of a given  function $f\in L^{p,\infty}(\Om)$ to $  L^\infty(\Om)$ as
\begin{equation}\label{dist_infty}
{\rm{dist}}_{L^{p,\infty} (\Om) }(f,L^\infty(\Om))=\inf_{g\in L^\infty(\Om)} \|f-g\|_{L^{p,\infty}(\Om)}.
\end{equation}
Note that, since $\|~\|_{p,\infty}$ is not a norm, ${\rm{dist}}_{L^{p,\infty} (\Om)}$ is just equivalent to a metric. In \cite{CS} is proved that
\begin{equation}\label{distlim}
{\rm{dist}}_{ L^{p,\infty}(\Om)  }(f, L^\infty(\Om))=\lim_{M \to +\infty}\|f-T_M f\|_{L^{p,\infty}(\Om)},
\end{equation}
where the truncation at level $M>0$ is defined as
\begin{equation}\label{T}
T_M  \,y=\frac{y}{|y|}\min\{|y|, M\}.
\end{equation}

Now,
let  $\vec{p}= (p_1,p_2,...,p_N)$ with $p_i>1$ for $i=1,...,N$, and let $\Omega$ be a bounded open subset of $\Rn$. As usual the anisotropic Sobolev space is the Banach space defined as

\[
W^{1,\vec{p}}(\Om)= \{ u\in W^{1,1} (\Om) : \de_{x_i} u\in L^{p_i} (\Om),i=1,...,N\}
\]
equipped with \[
\|u\|_{W^{1,\vec{p}}(\Om)} = \|u\|_{L^{1}(\Om)} +\sum_{i=1}^N  \|\de_{x_i} u\|_{L^{p_i}(\Om)}.
\]

It is well-known that in the anisotropic setting a Poincar\`e type
inequality holds true (see \cite{Fragala}). Indeed for every $u\in C_0^{\infty}(\Om)$ with $\Omega$ a bounded open set with Lipschitz boundary we have
\begin{equation}\label{dis poincare}
\|u\|_{L^{p_i}(\Om)} \leq C_P  \|\de_{x_i} u\|_{L^{p_i}(\Om)},\qquad \qquad i=1,...,N
\end{equation}
for a constant $C_P$ proportional to the width of $\Omega$ in the direction of $e_i$ and then to the diameter of $\Omega$. Moreover, for $u\in C_0^{\infty}(\mathbb{R^N})$ the following anisotropic Sobolev inequality holds true (see \cite{Ta})
\begin{equation}\label{Sob}
\|u\|_{L^{p,q}(\mathbb{R^N})} \le S_N\prod_{i=1}^N \| \de_{x_i} u\|_{L^{p_i}(\mathbb{R^N})}^{\frac 1N},
\end{equation}
where $S_N$ is an universal constant and  $p=\bar{p}^*$
and $q=\bar{p}$ whenever $\bar{p}<N$, where $\bar p$ is defined in \eqref{p bar} and
\begin{equation}\label{p star}
\bar{p}^*=\frac{N\bar{p}}{N-\bar p}.
\end{equation}
%Otherwise if $\bar p > N$ the function $u$ is in $L^\infty(\Omega)$ and if $\bar p=N$ we get that $e^{C|u|^{N'}}\in L^1_{\text{loc}}(\mathbb{R}^N)$ for every $C>0$.
Using the inequality between geometric and arithmetic mean we can replace the right-hand-side of \eqref{Sob} with $ \sum_{i=1}^N \| \de_{x_i} u\|_{L^{p_i} }$.
 In \cite{nostro} is proved a generalization of \eqref{Sob} to product of function. We recall an its simpler form that we will need in the following
\begin{equation}\label{Sob prod}
\left\|u\right\|_{L^{\bar p^*,\bar p}(\Omega)}\leq \widetilde{S}_N
\left\|\left(\prod_{i=1}^{N}\left|\de_{x_i} u\right|\right)^{1/N}\right\|_{L^{\overline{p}}(\Omega)},
\end{equation}
for a suitable universal constant $\widetilde{S}_N$.

When $\overline{p}<N$ and $\Omega$ is a bounded open set with Lipschitz boundary, {the space $W_{0}^{1,\overrightarrow{p}}(\Omega)= \overline{ C_0^\infty (\Om)}^{\sum_{i=1}^N \| \de_{x_i} u\|_{L^{p_i} }}$ is continued embedding into $ L^{q}(\Omega)$ for $q\in\lbrack1,p_{\infty}]$, with
$p_{\infty}:=\max\{\overline{p}^{\ast}, \max_ip_i \}$ as a consequence of \eqref{Sob} and \eqref{dis poincare}}.

\medskip

We recall the following useful lemma.

\begin{lemma}\label{Lemma tetai}\emph{(}see \cite[page 43]{KPS}\emph{)} Let $X$ be a rearrangement invariant space and let $
0\leq \theta _{i}\leq 1$ for $i=1,...,M,$ such that $\sum_{i=1}^{M}\theta _{i}=1$, then
\begin{equation*}
\left\| \prod_{i=1}^{M}|f_{i}|^{\theta _{i}}\right\| _{X}\leq
\prod_{i=1}^{M}\Vert f_{i}\Vert _{X}^{\theta _{i}} \quad \forall f_i\in X.  \label{holder
X}
\end{equation*}
\end{lemma}
\nc

\

\section{An useful a priori estimates}

In this  section \nc  we suppose that a weak solution $u$ \nc of the problem \eqref{P} exists and we prove the following a priori estimate under suitable assumptions on ${\rm{dist}}_{L^{\frac{N p'_i }{\bar p},\infty}(\Om)}(b_i, L^\infty(\Om))$, defined in \eqref{dist_infty}, for $i=1,\cdots,N$.

Since $\mathcal{F}$  belongs to $(W_0^{1,\vec{p}}(\Om))^*$, in what follows we can write it as

\begin{equation}\label{F}
\mathcal{F}=-\sum_{i=1}^N(f_i)_{x_i}, \mbox{ with }  f_i \in L^{p_i'}(\Omega)  \,\,\forall  i=1,...,N.
\end{equation}

\begin{lemma}\label{Lemma 1}Let us assume that  $\Omega$ is a  bounded Lipschitz domain, $p_{i}>1$ for $i=1,..,N$, $\bar{p}<N$, $(\mathcal H 1)-(\mathcal{H}4)$ are in force and let $u\in W_0^{1,\vec{p}}(\Om)$ be a weak solution of \eqref{P}. Then there exists a positive constant $d=d(N,\alpha,\vec{p})$ such that whenever
\begin{equation}\label{hpdistmain}
\max_i\left\{{ \emph{dist}}_{L^{\frac{N p'_i }{\bar p},\infty}(\Om)} (b_i, L^\infty(\Om))\right\} <d,
\end{equation}
the following uniform estimate holds
\begin{equation}\label{stima a priori1}
\sum_{i=1}^N\int_{\Omega}|u_{x_i}|^{p_i}dx\leq C,
\end{equation}
 where $C=C (\alpha,N,\overrightarrow{p},d,\|\mathcal{F}\|_{(W_0^{1,\vec{p}}(\Om))^*}).$
 \end{lemma}

\begin{proof} Using as test function $T_k u$ in \eqref{sol1}, by \eqref{ii2} and \eqref{b}  we get
\begin{align*}
 \alpha \sum_{i=1}^N \int_{\Omega_k}|u_{x_i}|^{p_i} dx + \int_\Om \mathcal G (x,u)T_k u \,dx  &\leq
\sum_{i=1}^N\int_{\Omega_k} |b_i(x)||u|^{\frac{\bar p}{p_i'}}|u_{x_i}| dx+\sum_{i=1}^N \int_{\Omega_k} |f_i||u_{x_i}| dx,
\end{align*}
where $\Omega_k=\{ x\in \Omega: |u(x)|\nc<k\}$. For all $M>0$, %it follows \nc
using \eqref{sign}, Young  and H\"{o}lder inequality  we have
\begin{align*}
\sum_{i=1}^N\int_{\Omega_k}|u_{x_i}|^{p_i} dx
\leq&
C\left(\sum_{i=1}^N\int_{\Omega_k} |u|^{\bar p} dx+\sum_{i=1}^N \|b_i(x)-T_M b_i\|^{p'_i}_{L^{\frac{N p'_i }{\bar p},\infty}(\Om)} \|T_k u\|_ {L^{\bar p^*,\bar p} (\Om)}^{\bar p}\right.\\
&+ \left.\sum_{i=1}^N\int_{\Omega_k} |f_i|^{p'_i} dx \right),
\end{align*}
where $C$ is a suitable positive constant  which can vary from line to line.

Denoting $ B = \prod_{i=1}^N \|(T_k u)_{x_i} \|_{L^{p_i}(\Om)}$, by Sobolev inequality \eqref{Sob} it follows

\begin{align}\label{somma}
\sum_{i=1}^N\int_{\Omega_k}|u_{x_i}|^{p_i} dx
\leq&
C\left(\int_{\Omega_k} |u|^{\bar p} dx+\sum_{i=1}^N \|b_i(x)-T_M b_i\|^{p'_i}_{L^{\frac{N p'_i }{\bar p},\infty}(\Om)} B^{\frac {\bar p}{N}}\right.\\
&+ \left.\sum_{i=1}^N\int_{\Omega_k} |f_i|^{p'_i} dx \right).\nonumber
\end{align}
Previous inequality gives us an estimate of the $j^{th}$ addendum of the sum at the left-hand side of \eqref{somma} as well. Then, elevating to the power $\frac{1}{N p_j}$, making the product on the left and right sides of \eqref{somma} we get

\begin{align*}
B^{\frac 1N}
\leq&
C \left[ \left(\int_{\Omega_k} |u|^{\bar p
} dx \right)^{\frac 1{\bar p}} + \sum_{i=1}^N \|b_i(x)-T_M b_i\|^{\frac {p'_i}{\bar p}}_{L^{\frac{N p'_i }{\bar p},\infty}(\Om)} B^{\frac 1 N} +\left( \sum_{i=1}^N\int_{\Omega_k} |f_i|^{p'_i} dx \right)^{\frac 1{\bar p}} \right].
\end{align*}

 At this point, assuming $d<\left( \frac{1}{CN}\right)^{\frac {\bar p}{p'_i}}  $ in \eqref{hpdistmain}  and using   \eqref{distlim}, we can fix now $M$ large enough in order  to have

\begin{align}\label{stimaB}
B^{\frac 1N}
\leq&
C \left[ \left(\int_{\Omega_k} |u|^{\bar p} dx \right)^{\frac 1{\bar p}} + \left( \sum_{i=1}^N\int_{\Omega_k} |f_i|^{p'_i} dx \right)^{\frac 1{\bar p}} \right].
\end{align}

Combining \eqref{somma} and \eqref{stimaB} we obtain
\begin{align}\label{lemmacompact}
\sum_{i=1}^N\int_{\Omega_k}& |u_{x_i}|^{p_i} dx \leq
C\left( \int_{\Omega_k}   |u|^{\bar p} dx + \sum_{i=1}^N \int_{\Omega_k}  |f_i|^{p'_i} dx\right),
\end{align}
for every $k>0$ and with $C$ positive costant independent of $u$ and $k.$
We prove now that previous estimate gives \eqref{stima a priori1}.  To this aim we follow the idea of  \cite[Lemma 2]{FGMZ}. We argue by contradiction and assume that there exists a sequence of functions $\{u_n\}_n \subseteq  W^{1,\vec{p}}_0(\Omega)$
 satisfying \eqref{lemmacompact} and such that
\[\| u_n \|:=\| u_n \|_{\SobpO}\rightarrow \infty\]
as $n\rightarrow \infty$.
For every $n\in \naturale$ and $\ee >0$, we set $k_n=\ee\|u_n\|$ so that by \eqref{lemmacompact}, for $i=1,...,N,$
\begin{equation*}\label{3.3ter}
\begin{split}
\left(\int_\Omega |\de_{x_i} ( T_{k_n}u_n)|^{p_i}
\dx x\right)^\frac{1}{p_i}
&\leqslant
C^\frac {1}{p_i}\left(
1+
\int_{\Omega}  |u_n|^{\bar p}\chi_{\{|u_n|<k_n\}} \dx x
\right)^\frac{1}{p_i},\\
\end{split}
\end{equation*}
where $T_{k_n}$ is defined in \eqref{T}. Since $\bar p<\bar p^*$, using Lemma \ref{Lemma tetai} with $X=L^{\bar p}(\Omega)$ and $\theta_i=\frac{\bar p}{p_i N}$, we get from the previous inequality
\begin{equation}\label{prodotto}
\begin{split}
 \left(\int_\Omega \prod^N_{i=1}  |\de_{x_i} ( T_{k_n}u_n)|^{\frac {\bar p}{N}}
\dx x\right)^\frac{1}{\bar p}
&\leqslant
\prod^N_{i=1}\left(\int_\Omega |\de_{x_i} ( T_{k_n}u_n)|^{p_i}
\dx x\right)^\frac{1}{N p_i}
\\
&\leqslant C ^\frac {1}{\bar p}\left(
1+
\int_{\Omega}  |u_n|^{\bar p}\chi_{\{|u_n|<k_n\}} \dx x
\right)^\frac{1}{\bar p}.\\
\end{split}
\end{equation}
We set
\[w_n=\frac{u_n}{\|u_n\|}\,.\]
so that, up to subsequence not relabeled,  there exists $\bar w \in\SobpO$ such that
$w_n\rightharpoonup \bar w$ weakly in $\SobpO$, $w_n\to \bar w$ strongly in $L^{\bar p}(\Om)$ and $w_n\to \bar w$ for a.e. in $ \Omega$.
Dividing both sides of
  \eqref{prodotto}
by $\|u_n\|^{\bar p}$  we have
\begin{equation}\label{stima vk ter}
 \int_\Omega  \prod^N_{i=1}  |\de_{x_i} ( T_{\ee
}w_n)|^{\frac {\bar p}{N} }
\dx x =\int_\Omega \frac{ \prod^N_{i=1}  |\de_{x_i} ( T_{k_n}u_n)|^{\frac {\bar p}{N} }}{\|u_n\|^{\bar p}  }
\dx x   \leqslant
C\left(\frac{1}{\|u_n\|^{ \bar p}}+
\int_{\Omega}  |w_n|^{\bar p}\chi_{\{|w_n|<\ee    \}} \dx x
\right).
\end{equation}
 Assume now that
\begin{equation}\label{e eccezionale ter}
|\{x\in\Omega:|\bar w(x)|=\ee\}|=0\,.
\end{equation}
In this case we have $\chi_{\{|w_n|<\ee\}}\to \chi_{\{|\bar w |<\ee\}}$ a.e.\ in $\Omega$ and hence $w_n\,\chi_{\{|w_n|<\ee\}}\to \bar w\,\chi_{\{|\bar w|<\ee\}}$
strongly in $L^{\bar p}(\Omega)$.  So,  since  $T_\ee w_n\rightharpoonup T_\ee \bar w$ weakly in $\SobpO$ and $T_\ee w_n\to T_\ee \bar w$ strongly in $L^{\bar p}(\Omega)$, letting $n\to+\infty$ in \eqref{stima vk ter},  using the semicontinuity of the norm with respect to weak convergence in $\SobpO$, we arrive to the following estimate
\begin{equation}\label{stima v ter}
  \int_\Omega { \prod^N_{i=1}  |\de_{x_i} ( T_{\ee}\bar w)|^{\frac {\bar p}{N} }}
\dx x  \leqslant C \int_{\Omega}|\bar w|^{\bar p}\chi_{\{|\bar w|<\ee\}} \dx x.
\end{equation}
Using Sobolev  inequality \eqref{Sob prod} and H\"{o}lder inequality by \eqref{stima v ter}  we get
\[\ee^{\bar p } \,|\{x\in \Omega:|\bar w|\geqslant\ee\}|^{\frac{\bar p }{\bar p *}}\leqslant C\,\ee^{\bar p } \,|\{x\in \Omega:0<|\bar w|<\ee\}| \,.\]
Passing to the limit as $\ee\downarrow 0$, we deduce
\[|\{x\in \Omega:|\bar w|>0\}|=0\,,\]
that is, $\bar w(x)=0$ a.e. Note that previous equality has been obtained assuming \eqref{e eccezionale ter}. Nevertheless,  the set of values $\ee>0$ for which \eqref{e eccezionale ter} fails is at most countable.
This means $w_n\rightharpoonup 0$ weakly in $\SobpO$. On the other hand, at this point we can use again the same argument above to obtain that $w_n\to0$ strongly in $\SobpO$,  and this gives the contradiction, since by definition $\|w_n\|=1$ for every $n\in \naturale$.
\medskip
\end{proof}

\section{Regularity results}
In this section we study the regularity of weak solutions of problem \eqref{P}. Here and in the following we shall assume notation \eqref{F} being in force. When  the datum $f_i\in L^{p_i}(\Omega)$ for $i=1,..,N$ with $\bar p<N$, as follows by the anisotropic Sobolev embedding, a solution $u$ belongs to $L^{\overline{p}^*}(\Omega)$, where $\bar{p}^*$ is defined in \eqref{p star}. Otherwise if ${f}_i\in L^{s_i}(\Omega)$ with $s_i>p_i$ for $i=1,..,N$, the summability of $u$ improves. In order to study the higher summability of $u$,  the following minimum
\begin{equation}\label{mu}
\mu=\min_{i}\left\{ \frac{s_i}{p_i'}\right\}
\end{equation}
first introduced in \cite{BMS}, plays a crucial role. In \cite[Theorem 2.1 and Remark 5.3]{nostro}  the following Stampacchia type regularity result is proved (see \cite{S} as classical reference).

\begin{theorem}\label{mainTHM} Let $\Omega$ be a  bounded Lipschitz domain, $p_{i}>1$ for $i=1,..,N$, $\bar{p}<N$ and let $s_1,\cdots,s_N$ be such that
\begin{equation*}\label{hp-ri}
1<\mu< \frac{N}{\bar p},
\end{equation*}
 where $\mu$ is defined in \eqref{mu}. Assume that $(\mathcal H 1)-(\mathcal{H}3)$  are fulfilled and  $f_i\in L^{s_i} (\Om)$ for $i=1,..,N$. There exists a positive constant  $d=d(\vec r,N, \alpha, \vec{p})$  such that if
\[
\max_i\left\{{ \emph{dist}}_{L^{\frac{N p'_i }{\bar p},\infty}(\Om)} (b_i, L^\infty(\Om))\right\} <d
\]
and  $u\in W^{1,\vec p}_0(\Om)$
is a weak solution to \eqref{P}, then
\begin{equation*}\label{rem1}
u\in L^{s}(\Om) \quad \text{with }s=\max\{(\mu \bar p)^*,\mu p_{\max}\}.
\end{equation*}
where
\begin{equation*}\label{s}
(\mu\bar p)^*= \frac{N\mu \bar p}{N-\mu\bar p}.%\,\qquad \mbox{ and }\qquad\mu=\min_{i}\left\{ \frac{r_i}{p_1}\right\}
\end{equation*}
\end{theorem}

\medskip

Without lower order terms in \cite{BMS} the authors have proved that the boundedness of a weak solution of Dirichlet problems is guaranteed under the assumption
$\mu> \frac{N}{\bar p},$
where $\mu$ is defined in \eqref{mu}. However if $\mathcal{B}_i\not\equiv0$  for $i=1,\cdots,N$ the boundedness is not assured assuming that \eqref{hpdistmain} is in force as showed in Example 4.8 of \cite{GGM} (when $p_i=2$  for $i=1,\cdots,N$). The smallness of $\|b_i\|_{L^{\frac{Np_i'}{\bar p},\infty}}$ for $i=1,\cdots,N$ neither is sufficient to get boundedness, as showed in Example 2.3 of \cite{nostro}.

In order to get the boundedness of solutions,  we need to improve the summability of data and of the coefficients $b_i$.

 \begin{lemma}\label{L infty}

Let  $\Omega$ be a  bounded Lipschitz domain, $p_{i}>1$ for $i=1,..,N$, $\bar{p}<N$,
 and let $s_1,\cdots,s_N$ be such that
  \begin{equation}\label{alpha}
  \mu >\frac{N}{\bar p}
  \end{equation}
 where $\mu$ is defined in \eqref{mu}. Assume that  $(\mathcal H 1)-(\mathcal{H}3)$  are fulfilled and  $b_i,f_i\in L^{s_i} (\Om)$ for $i=1,..,N$. Then  every weak solution $u$ of problem \eqref{P} is bounded and  there exists a positive constant $C= C (\alpha,\Omega, N,\overrightarrow{p},\|{b}_i\|_{L^{s_i}(\Omega) } , \|{f}_i\|_{L^{s_i}(\Omega) } )$ such that
 \[
 \|u\|_{\infty}\leq C.
 \]
\end{lemma}
\begin{proof}
The proof is quite standard, and for the convenience of the reader we give some details. So, let $u$ be a weak solution of problem \eqref{P}. For $k>0$, we use as a test function $G_ku:= u-T_k u$ in \eqref{sol1} and by \eqref{ii1}, \eqref{b} we obtain

\begin{align*}
 \alpha \sum_{i=1}^N \int_{\Omega}|(G_ku)_{x_i}|^{p_i} dx + \int_\Om \mathcal G (x,u)G_k u dx  &\leq
\sum_{i=1}^N\int_{\Omega} |b_i(x)||u|^{\frac{\bar p}{p_i'}}|(G_k u)_{x_i}| dx+\sum_{i=1}^N \int_{\Omega} |f_i||(G_ku)_{x_i}| dx.
\end{align*}
Denoting $A_k=\{ x\in \Omega: |u(x)|>k\}$, we note that by Lemma \ref{Lemma 1} for every $k\geq 0$ we have:
\begin{equation}\label{indip-u}
|A_k|\leq \frac{\|u\|_{L^1(\Om)}}{k} \leq \frac{C_0}{k},
\end{equation}
where $C_0$ is a positive constant indepedent on $u$. We stress that \eqref{hpdistmain} is obviously satisfied when $b_i\in L^{s_i}(\Omega)$.

Using \eqref{sign} and Young inequality we have
\begin{align*}
&&\alpha \sum_{i=1}^N \int_{\Omega}|(G_ku)_{x_i}|^{p_i} dx \leq C
\sum_{i=1}^N\int_{A_k} |b_i(x)|^{p_i'} |G_ku|^{\bar p} dx +C \sum_{i=1}^N\int_{A_k} |b_i(x)|^{p_i'} k^{\bar p} dx \\
&&+ \varepsilon \sum_{i=1}^N  \int_{\Omega}  |(G_ku)_{x_i}|^{p_i} dx+C\sum_{i=1}^N \int_{A_k} |f_i|^{p_i'} dx  + \varepsilon \sum_{i=1}^N  \int_{A_k}  |(G_ku)_{x_i}|^{p_i} dx,\\
\end{align*}
for every $\varepsilon >0$ and a suitable $C>0$. Choosing $\varepsilon =\varepsilon (\alpha, \vec{p})$ small enough we have

\begin{align*}
&&\sum_{i=1}^N \int_{\Omega}|(G_ku)_{x_i}|^{p_i} dx \leq C \left(
\sum_{i=1}^N\int_{A_k} |b_i|^{p_i'} |G_ku|^{\bar p} dx + \sum_{i=1}^N\int_{A_k} |b_i(x)|^{p_i'} k^{\bar p} dx + \sum_{i=1}^N \int_{A_k} |f_i|^{p_i'} dx \right)\\ \\
&& \leq    C\left[ \sum_{i=1}^N \|b_i\|^{p'_i}_{L^{\frac{N p'_i }{\bar p}}(A_k )} \|G_k u\|_ {L^{\bar p^*} (\Om)}^{\bar p} +\sum_{i=1}^N  \int_{A_k} \left(|b_i(x)|^{p_i'} k^{\bar p}  + |f_i|^{p_i'}\right) dx  \right],\\
\end{align*}
where, here and below, $C$ is a suitable positive constant  which can vary from line to line.

%%%%%%%%INIZIO 26 Aprile 2023
As before, denoting $ B = \prod_{i=1}^N \|(G_k u)_{x_i} \|_{L^{p_i}(A_k)}$, by Sobolev inequality \eqref{Sob} it follows

\begin{align}\label{somma2}
\sum_{i=1}^N\int_{\Omega}|(G_ku)_{x_i}|^{p_i} dx
\leq&
C\left[\sum_{i=1}^N \|b_i\|^{p'_i}_{L^{\frac{N p'_i }{\bar p}}(A_k)} B^{\frac {\bar p}{N}}
+\sum_{i=1}^N  \int_{A_k} \left(|b_i(x)|^{p_i'} k^{\bar p}  + |f_i|^{p_i'}\right) dx \right].
\end{align}
Previous inequality gives us an estimate of the $j^{th}$ addendum of the sum at the left-hand side of \eqref{somma2} as well. Then, elevating to the power $\frac{1}{N p_j}$, making the product on the left and right sides of \eqref{somma2} we get

\begin{align*}
B^{\frac 1N}
\leq&
C \left[  \sum_{i=1}^N \|b_i(x)\|^{\frac {p'_i}{\bar p}}_{L^{\frac{N p'_i }{\bar p}}(A_k)} B^{\frac 1 N} +\left( \sum_{i=1}^N  \int_{A_k} (|b_i(x)|^{p_i'} k^{\bar p}  + |f_i|^{p_i'}) dx \right)^{\frac 1{\bar p}} \right].
\end{align*}

 At this point, we can choose $k=k(N, \vec{p}, b_i, \alpha, C)>1$  large enough in order  to have $|A_k|$ small and so by absolutely continuity of the Lebesgue norm we have

\begin{align}\label{stimaB2}
B^{\frac 1N}
\leq&
C_1 \left[  \left( \sum_{i=1}^N  \int_{A_k} (|b_i(x)|^{p_i'} k^{\bar p}  + |f_i|^{p_i'}) dx \right)^{\frac 1{\bar p}} \right]
\end{align}
for suitable $C_1>0$. Combining \eqref{somma2} and \eqref{stimaB2} we obtain

\begin{equation}\label{dicas1}
\sum_{i=1}^N\int_{\Omega} |(G_k u)_{x_i}|^{p_i} dx \leq
C_2\left(  \sum_{i=1}^N  \int_{A_k} (|b_i(x)|^{p_i'} k^{\bar p}  + |f_i|^{p_i'}) dx\right).
\end{equation}

As before, we note that previous inequality gives us an estimate of the $j^{th}$ addendum of the sum at the left-hand side of \eqref{dicas1} as well. Then, elevating to the power $\frac{1}{N p_j}$, making the product on the left and right sides of \eqref{dicas1} and using Sobolev inequality, we get
\begin{eqnarray*}
\| G_k u\|_{L^{\bar{p}^*}(\Om)}  &&\leq C_3\prod_{j=1}^N  \left(  \sum_{i=1}^N  \int_{A_k} (|b_i(x)|^{p_i'} k^{\bar p}  + |f_i|^{p_i'}) dx\right)^{\frac{1}{N p_j}} \\
& &\leq  C_3   \left(  \sum_{i=1}^N  \int_{A_k} (|b_i(x)|^{p_i'} k^{\bar p}  + |f_i|^{p_i'}) dx\right)^{\frac{1}{\bar{p}}} \\
&& \leq  C_3 k   \left(  \sum_{i=1}^N ( \|b_i\|^{p_i'} _{s_i}  + 	\|f_i\|_{s_i}^{p_i'}) \,\,|A_k|^{1-\frac{p'_i}{s_i}}   \right)^{\frac{1}{\bar{p}}}.    \\
\end{eqnarray*}

We fix now $k_0>1 $  such that $|A_{k_0}|< 1$. Note that in view of \eqref{indip-u} we can fix such $k_0$ independent on $u$.

Applying H\"{o}lder inequality at the left hand side of previous inequality, for every $k\geq k_0 $ we get

\begin{equation}\label{indep-u2}
\int_\Om |G_k u|\, dx \leq C \, k |A_k|^{\frac {\eta}{\bar p }+1 -\frac 1{\bar {p}^*}}\qquad  \mbox{ where }\quad \eta = \min_i \{ 1-\frac{p'_i}{s_i}\}.
\end{equation}
Observe now that the function
\[
g(k):=\int_\Om |G_k u| \,dx
\]
is a non-negative and decreasing function such that $g'(k)= -|A_k|$. Hence we can rewrite  \eqref{indep-u2} as
\[
\left(\frac{g(k)}{k}\right)^{a} \leq -C g'(k) \qquad  \mbox{ where }\quad   a := \frac{\bar{p}\bar{p}^*}{\eta \bar{p}^* + \bar{p}\bar{p}^*-\bar{p}}
\]
and by \eqref{alpha}, it holds $(1-a) >0$.
Let us consider now every value of  $k> k_0$ such that $g(k)\not= 0.$ By previous inequality we get

\[
{k}^{-a} \leq -C g'(k) g(k)^{-a} = -C (g(k) ^{(1-a)})'.
\]

Integrating previous inequality with respect to $k$ from $k_0 $ to $k$ we get

\[
{k}^{1-a}- {k_0}^{1-a} \leq -C (g(k)^{1-a}- g(k_0) ^{(1-a)}).
\]
This implies

\[
g(k)^{1-a} \leq    C g(k_0) ^{(1-a)}   - {k}^{1-a}+ {k_0}^{1-a}.
\]

It is obvious at this point that, since $g(k_0)\leq \|u\|_{L^1(\Om)}$, there exists a value $\bar{k} $ (independent in $u$) such that $g(\bar {k}) = 0 $ and the thesis follows.

\end{proof}

\section{Existence result}

In this section we analyze the existence of weak solutions of problem \eqref{P}.
It is well known that, if the operator in problem \eqref{P} is coercive, then a solution exists in $W_0^{1,\vec{p}}(\Om)$.  Unfortunately, under our assumptions  the coercivity for the involved operator in problem \eqref{P} is not guaranteed.
Another difficulty is due to the
singularity of coefficients $b_i$ in the lower order term. Indeed in the Marcinkiewicz space  $L^{\frac{N\,p'_i} {\bar p},\infty}(\Om)$, which is slightly larger than Lebesgue space $L^{\frac{N\,p'_i} {\bar p}}(\Om)$, the bounded functions are not dense and the norm is not absolutely continuous (i.e. a function can have large norm even if restricted to a set with small measure).

\begin{theorem}\label{TH E1}
Let us assume that  $\Omega$ is a  bounded Lipschitz domain, $p_{i}>1$ for $i=1,..,N$, $\bar{p}<N$ and ($\mathcal{H}1$)-($\mathcal{H}4$)  are in force. There exists  a positive constant  $d=d(N, \alpha, \vec{p})$  such that if \eqref{hpdistmain} holds,
then there exists
at least a weak solution to problem (\ref{P}).
\end{theorem}

\medskip

For example in our  prototype \eqref{model 1} we can consider  as coefficient of the lower order term the function

\[
\beta_i(x)=\frac{\kappa_i}{|x|^{\gamma_i}}+\beta_i^0(x)
\]  with $\kappa_i$ suitable small constant, $\beta_i^0 \in L^{\infty}(\Om)$ and
$\gamma_i=\frac{\bar p}{p_i'}$ for $i=1,\cdots,N$. Indeed, in this case, an easy computation shows that
 $
{\rm{dist}}_{L^{\frac{Np'_i}{\bar p},\infty}(\Omega)}   (\beta_i, L^\infty(\Om))=  |\kappa_i|\omega_N^{\frac{\gamma_i}{N}}$.

\medskip

Our proof is detailed in the next subsection.

%%%%%%%%%%%%%%FIME LEMMA 3.3
\medskip

\subsection{Proof of Theorem \ref{TH E1}}\label{dim E1}

 We split the proof in three steps.\nc

\textit{Step 1: Existence when $b_i\in L^\infty(\Omega)$ and $f_i \in C_c^\infty(\Omega)$ for $i=1,\cdots,N$.}

\medskip
We stress that also \nc in this case the operator $\sum_{i=1}^N  \left[ -\de_{x_i}  (\mathcal{A}_i(x,\nabla u)+\mathcal B_i(x,u)    )\right]+\mathcal G(x,u)$ can be not coercive. To overcome this difficulty we use the  boundedness result proved in Lemma \ref{L infty} and so we suppose that a solution $u$ there exists. Using Lemma \ref{L infty}, it follows that
\begin{equation}\label{Stima u}
\|u\|_\infty\leq M,
\end{equation}
where $M$ is a constant depending only on the data of the equation.

Now we define
$$\widetilde{\mathcal{B}}(x,s)=\left\{
\begin{array}[c]{ll}%
\mathcal{B}(x,s) &\hbox { if } |s|\leq M,
 \\
 \mathcal{B}(x,M)& \hbox{ if } |s|>M.
\end{array}
\right.
$$
Problem \eqref{P} with $\mathcal{B}(x,s)=\widetilde{\mathcal{B}}(x,s)$, $b_i\in L^\infty(\Omega)$ and  smooth data \nc has a solution $u$,  since \nc the operator is coercive. Moreover this solution  $u$ verifies estimate \eqref{Stima u}, because $\mathcal{B}(x,s)$ verifies  \eqref{b}. It follows that $\widetilde{\mathcal{B}}(x,s)=\mathcal{B}(x,s)$ and we have proved the existence of a solution to problem \eqref{P} when $b_i\in L^\infty(\Omega)$ and  data \nc are smooth.

\medskip

\noindent \textit{Step 2: Existence when $b_i\in L^\infty(\Omega)$ and $\mathcal{F}\in (W_0^{1,\overrightarrow{p}}(\Omega))^*$.}

We consider the following sequence of approximate problems
\begin{equation}\label{Pn}
\left\{
\begin{array}[c]{ll}%
-\sum_{i=1}^N   \de_{x_i}  \left[\mathcal{A}_i(x,\nabla u_n)+\mathcal B_i(x,u_n)    \right]
+ \mathcal G(x,u_n )
= - \sum_{i=1}^N(f^n_i)_{x_i} &\hbox { in } \Omega,
\\
& \\
u_n=0  & \hbox{ on } \partial \Omega,
\end{array}
\right.
\end{equation}
where $f^n_i\in C_c^\infty(\Omega)$ such that
\begin{equation}\label{f smooth}
f^n_i\rightarrow f_i \quad \text{ in } L^{p_i'}(\Omega),
\quad \|f^n_i\|_{ L^{p_i'}(\Omega)}\leq\|f_i\|_{ L^{p_i'}(\Omega)} \text{ for }i=1,\cdots,N.
\end{equation}
Step 1  assure  the existence of a solution $u_n$ of problem \eqref{Pn}. We stress that estimate \eqref{stima a priori1} holds for problem  \eqref{Pn} when  $b_i\in L^\infty(\Omega)$ for $i=1,\cdots, N$.

  We get that the sequence $\{u_n\}_n$ is bounded in $W_0^{1,\overrightarrow{p}}(\Omega)$ by a constant independent of $n$. Then there exists a function $u \in W_0^{1,\overrightarrow{p}}(\Omega)$ such that (up to a subsequence denoted again by $u_n$)
\begin{equation}\label{weak un}
u_n\rightharpoonup u \text{ weakly in } W_0^{1,\overrightarrow{p}}(\Omega),
\end{equation}
\begin{equation}\label{forte un}
u_n\to u  \text{ in } L^q (\Omega) \text{ for } q<p_\infty
\end{equation}
and
\begin{equation}\label{qo un}u_n\to  u  \text{ a.e. } \Omega.
\end{equation}

To take into account the terms $\mathcal{A}_i(x,\nabla u_n)$ we need the following lemma, which is the anisotropic version of Lemma  2.2 of \cite{L}. Note that if $A_i$ is strongly monotone, the claim is obvious.
\begin{lemma}\label{conv q.o.}
Assume  $(\mathcal H 1)$ be in force. Let us suppose $u_n,u \in W_0^{1,\overrightarrow{p}}(\Omega)$,  $u_n\rightharpoonup u $ weakly in $W_0^{1,\overrightarrow{p}}(\Omega)$ and
\begin{equation}\label{Ip lemma}
\int_\Omega \left(\mathcal{A}_i(x,\nabla u_n)-\mathcal{A}_i(x,\nabla u)\right) \left[\partial_{x_i}u_n-\partial_{x_i}u\right] \,dx \rightarrow0
\end{equation}
for $i=1,\cdots,N$. Then for $i=1,\cdots,N$
\begin{equation}\label{conv qo grad}
\partial_{x_i}u_n\rightarrow \partial_{x_i} u \quad \text{ a.e. in } \Omega.
\end{equation}
\end{lemma}
\begin{proof} The proof is standard and follows
as in Lemma  2.2 of \cite{L}. For the convenience of the reader we give some details.  Let us denote  for $i=1,\cdots,N$
$$Q_n^i(x):=\left(\mathcal{A}_i(x,\nabla u_n(x))-\mathcal{A}_i(x,\nabla u(x))\right) \left[\partial_{x_i}u_n(x)-\partial_{x_i}u(x)\right].$$
Since \eqref{ii3} and \eqref{Ip lemma} we get (up to a subsequence) that
$$ Q_n^i(x)\rightarrow0 \text{ for every }x\in\Omega\setminus \Omega_0$$
with $|\Omega_0|=0$. Moreover by our assumptions we get (up to a subsequence) $u_n(x)\rightarrow u(x)$ for $x\in\Omega\setminus\Omega_0$ (using that $W_0^{1,\overrightarrow{p}}(\Omega)$ is compactly  embedded in $L^q(\Omega)$ for $q<p_\infty$). Let us suppose $x\in \Omega\setminus \Omega_0$ and let us denote  $\partial_{x_i}u_n(x)=\xi_n^i$ and $\partial_{x_i}u(x)=\xi^i$  and by $\bar\xi^i$ one of the limit of $\xi_n^i$ for $i=1,\cdots,N$. By \eqref{ii2} and \eqref{ii1} we have
\begin{equation}\label{E51}
Q_n^i(x)\geq \alpha |\xi_n^i|^{p_i}-\beta_i|\xi_n^i|^{p_i-1}|\xi^i|-\beta_i
|\xi_n^i||\xi^i|^{p_i-1}.
\end{equation}
If we assume $|\bar\xi^i|=+\infty$ then by \eqref{E51} we obtain $Q_n^i(x)\rightarrow +\infty$, which is a contradiction (see \eqref{Ip lemma}). Then $|\bar\xi^i|<+\infty\quad \forall i=1,\cdots,N$. Using the continuity of $\mathcal{A}_i$ with respect to $\xi$ we conclude that
$$\left(\mathcal{A}_i(x,\bar \xi)-\mathcal{A}_i(x,\xi)\right) \left[\bar \xi^i-\xi^i\right]=0\quad \forall i=1,\cdots,N.$$
So assumption \eqref{ii3} yields then $\bar \xi= \xi$, i.e. $\partial_{x_i}u_n(x)\rightarrow \partial_{x_i}u(x)\quad \forall i=1,\cdots,N$ for every $x\in\Omega\setminus \Omega_0$.
\end{proof}

\medskip

Now we prove that \eqref{Ip lemma} is in force.  We use $\varphi= u_n-u$ as test function in the variational formulation of the  approximating problems \nc \eqref{Pn}:
 \begin{equation*}\label{testPncase1}
 \Si \int_\Om   \left[\mathcal{A}_i(x,\nabla u_n)+   \mathcal B_i(x,u_n)    \right] \dei (u_n-u)    \, dx+\int_{\Omega}\mathcal G(x,u_n)(u_n-u)    \, dx = \sum_{i=1}^N\int_\Omega f_i^n\partial_{x_i}(u_n-u)\, dx.
\end{equation*}
 Adding and subtracting in the previous equality the term
$\sum_{i=1}^N\int_\Omega \mathcal{A}_i(x,\nabla u) \partial_{x_i}(u_n-u)$, by \eqref{ii1}, \eqref{weak un} and \eqref{f smooth} we easily obtain that
$$\sum_{i=1}^N\int_\Omega \mathcal{A}_i(x,\nabla u) \partial_{x_i}(u_n-u)\rightarrow0 \quad \text{ and } \quad \sum_{i=1}^N\int_\Omega f_i^n\partial_{x_i}(u_n-u)\, dx\rightarrow0.$$ \nc
Then recalling again \eqref{weak un} in order to prove \eqref{Ip lemma} it is enough to show that
\begin{equation}\label{con forte Bi}
\mathcal{B}_i(x,u_n)\rightarrow \mathcal{B}_i(x,u) \text{ in }L^{p'_i}(\Omega)\quad  \forall i=1,\cdots,N,
\end{equation}
\begin{equation}\label{con forte G}
\mathcal{G}(x,u_n)\rightarrow \mathcal{G}(x,u) \text{ in }L^{p'_{\infty}}(\Omega).
\end{equation}
To this aim we observe that by continuity of $\mathcal{B}_i$ with respect to $s$ and by \eqref{qo un} we get $\mathcal{B}_i(x,u_n)\rightarrow \mathcal{B}_i(x,u)$ a.e. in $\Omega$.  Now \eqref{b},  H\"{o}lder inequality, \nc Sobolev inequality \eqref{Sob} and \eqref{stima a priori1} yields that for every measurable set $\Om'\subset \Om$
\begin{equation*}\label{33}
\begin{split}
\int_{\Omega'}|\mathcal{B}_i(x,u_n)|^{p'_i}\, dx&\leq \|b_i\|^{p'_i}_\infty\int_{\Omega'}|u_n|^{\bar p}\, dx
\leq \|b_i\|^{p'_i}_\infty\left(\int_{\Omega'}|u_n|^{\bar p^*}dx\right)^{\frac {\bar p}{\bar p^*}}|\Omega'|^{1-\frac {\bar p}{\bar p^*}}\\
&\leq C \|b_i\|^{p'_i}_\infty |\Omega'|^{1-\frac {\bar p}{\bar p^*}},
\end{split}
\end{equation*}
where $C$ is independent of $n$. By Vitali convergence Theorem we conclude that \eqref{con forte Bi} holds.

On the other hand  by continuity of $\mathcal{G}$ with respect to $s$ we get $\mathcal{G}(x,u_n)\rightarrow \mathcal{G}(x,u)$ a.e. in $\Omega$, by \eqref{qo un}.  Now \eqref{b1},  H\"{o}lder inequality, Sobolev inequality \eqref{Sob} and \eqref{stima a priori1} yields that for every measurable set $\Om'\subset \Om$
$$\int_{\Omega'}|\mathcal{G}(x,u_n)|^{p'_{\infty}}\, dx\leq \widetilde{\mu}\int_{\Omega'}|u_n|^{\gamma p'_{\infty} }\, dx
\leq \mu\left(\int_{\Omega'}|u_n|^{ p_{\infty} }\, dx\right)^{\frac{\gamma}{p_{\infty}-1}}|\Omega'|^{1-\frac{\gamma}{p_{\infty}-1}}
\leq C |\Omega'|^{1-\frac{\gamma}{p_{\infty}-1}},$$
where $C$ is independent of $n$. Also in this case  we conclude that \eqref{con forte G} holds by Vitali convergence Theorem.

Then \eqref{Ip lemma} is in force, we can apply Lemma \ref{conv q.o.} and   convergence \nc \eqref{conv qo grad} follows for $i=1,\cdots,N$. Then
\[
\mathcal{A}_i(x,\nabla u_n)\rightarrow \mathcal{A}_i(x,\nabla u) \quad \text{ a.e. } \Omega,
\]
 because of continuity of $\mathcal{A}_i$ with respect to $\xi$. Moreover by \eqref{ii1} and \eqref{stima a priori1} we get that $\mathcal{A}_i(x,\nabla u_n)$ is bounded in $L^{p'_i}(\Omega)$ and then
\begin{equation}\label{con deb Ai}
\mathcal{A}_i(x,\nabla u_n)\rightharpoonup \mathcal{A}_i(x,\nabla u) \text{ weakly in } L^{p'_i}(\Omega).
\end{equation}
Taking $\phi \in W_0^{1,\overrightarrow{p}}(\Omega)$ as test function in \eqref{Pn} and using \eqref{con deb Ai}, \eqref{con forte Bi},\eqref{con forte G} and \eqref{f smooth} we can pass to the limit obtaining that $u$ is a solution of problem \eqref{P} when $b_i\in L^\infty(\Omega)$ and $\mathcal{F}\in (W_0^{1,\overrightarrow{p}}(\Omega))^*$.

\medskip

\noindent\textit{Step 3  : Existence dropping assumptions that $b_i\in L^\infty(\Omega)\quad \forall i=1,\cdots,N$.}

In order to remove the assumptions on $b_i\in L^{\infty }(\Omega)$, we set for each $n\in \mathbb N$ and $i=1,...,N$  for almost every $x\in \Omega$, \nc
\begin{equation*}
\vartheta^i_n(x)=
\left\{
           \begin{array}{ll}
            \frac{T_n b_i(x)}{b_i(x)} & \hbox{ if } b_i(x)\neq0  \\
             1  & \hbox{  if }b_i(x)=0  \\
           \end{array}
         \right.
\end{equation*}
and we consider the following  sequence of approximating problems:
\begin{equation}\label{Pn2}
\left\{
\begin{array}[c]{ll}%
-\sum_{i=1}^N   \de_{x_i} \left[ \mathcal{A}_i(x,\nabla u_n)+\mathcal \vartheta^i_n(x) \mathcal B_i(x,u_n) \right]
+ \mathcal G(x,u_n )
=   \mathcal{F} &\hbox { in } \Omega,
\\
& \\
u_n=0  & \hbox{ on } \partial \Omega.
\end{array}
\right.
\end{equation}

Applying the previous Step 2  with $ \vartheta^i_n(x) b_i (x)\in L^\infty(\Om)$ in place of $b_i$,  for every $n\in \naturale$,   we find a solution $u_n\in \SobpO$  to problem \eqref{Pn2}. Moreover,  by Lemma \ref{Lemma 1}, using estimate \eqref{stima a priori1} we get that the sequence $\{u_n\}_n$
is bounded in $\SobpO$ by a constant independent of $n$. \nc Then, unless to pass to  subsequences not relabeled there exists  $u \in \SobpO$ such that \eqref{weak un}, \eqref{forte un} and \eqref{qo un} hold.

We shall conclude our proof showing that such function $u$ solves problem \eqref{P}. We emphasize that in our assumptions the compactness of the sequence $\mathcal B_i(x, u_n )$ could fail (see Remark \ref{ex1}).

 To this end we use an idea contained in \cite{FGMZ}.
In the rest of our proof we let for simplicity   $\eta(t):=\arctan t$.  Obviously, $\eta \in C^1(\mathbb R)$, $|\eta (t)|\le |t|$ and $0 \le \eta^\prime(t) \le 1$ for all $t \in \mathbb R$. In particular,
$\eta$ is Lipschitz continuous in the whole of $\mathbb R$ and therefore
\[
u_n,u\in \SobpO \quad \Longrightarrow \quad \eta(u_n-u)\in \SobpO\,.
\]
Moreover, since $\eta(0)=0$ we have
\begin{equation}\label{23}
\eta(u_n-u) \rightharpoonup 0    \qquad \text{in $\SobpO$ weakly}\,.
\end{equation}

Testing equation in \eqref{Pn2} by the function $\varphi_n= \eta(u_n-u)$, we get
 \begin{equation}\label{ecco}
 \Si \int_\Om   \left[\mathcal{A}_i(x,\nabla u_n)+ \mathcal \vartheta^i_n(x)  \mathcal B_i(x,u_n)    \right] \dei \varphi_n    \, dx +\int_{\Omega}\mathcal G(x,u_n)\varphi_n    \, dx=  \langle \mathcal{F}\, , \varphi_n \rangle
\end{equation}

As before,  we add and subtract the term
$\sum_{i=1}^N\int_\Omega \mathcal{A}_i(x,\nabla u) \partial_{x_i}\varphi_n$ in previous equality.

In view of \eqref {ii1} and \eqref{b1}, by \eqref{23} we obviously  have

\begin{equation*}\label{con_Ai}
\lim_{n\rightarrow \infty} \sum_{i=1}^N\int_\Omega \mathcal{A}_i(x,\nabla u) \partial_{x_i}\varphi_n\rightarrow0, \quad\qquad   \lim_{n\rightarrow \infty}
 \langle \mathcal{F}\, , \varphi_n \rangle  =  0.
\end{equation*}
%and
%\begin{equation}\label{con_F}
%\sum_{i=1}^N\int_\Omega f_i^n\partial_{x_i}\varphi_n\, dx\rightarrow0.
%\end{equation}
Moreover also in this case  we get \eqref{con forte G} and so by \eqref{23} we obtain
\begin{equation*}\label{con_G}
\int_\Omega \mathcal{G}(x, u_n) \varphi_n\, dx\rightarrow0.
\end{equation*}
At this point  it suffices to show that (up a subsequence)
\begin{equation}\label{compattezza}
  \mathcal \vartheta^i_n(x)  \mathcal B_i(x,u_n) \rightarrow
 \mathcal B_i(x,u)  \,  \text{ strongly in }L^{p'_i}(\Omega) \quad \forall i=1,\cdots,N.
\end{equation}
We preliminary observe that  combining  \eqref{qo un} with the property that $\vartheta^i_n\rightarrow 1$ as $n\rightarrow \infty$, we have
\begin{equation}\label{Vitalin1}
\mathcal \vartheta^i_n(x)  \mathcal B_i(x,u_n) \eta' (u_n-u)=\frac {  \mathcal \vartheta^i_n(x)  \mathcal B_i(x,u_n)  }{1+|u_n-u|^2} \rightarrow
 \mathcal B_i(x,u)  \qquad \text{a.e.\ in }\Omega\,.
\end{equation}
Moreover,  using \eqref{b} we note that for every measurable set $\Om'\subset \Om$ it holds
\begin{equation}\label{Vitalin2}
\int_{\Om'}  \left[ \frac { \mathcal \vartheta^i_n(x)  \mathcal B_i(x,u_n)  }{1+|u_n-u|^2}\right]^{ p'_i}dx\le C\int_{\Om'}
\left[ b_i(x)^{p'_i} |u|^{\bar p} +\frac {b_i(x)^{p'_i} |u_n-u|^{\bar p} }{1+|u_n-u|^2}\right]dx,
\end{equation}
for some positive constant $C=C(\bar p)$. Hence, if $1<\bar p\le 2$,  \eqref{compattezza} immediately follows by Vitali convergence Theorem combining \eqref{Vitalin1} and \eqref{Vitalin2}. For $\bar p>2$ we choose $s$ satisfying
\[
\frac {\bar p^\ast}{\bar p} < s < \frac {\bar p^\ast}{\bar p-2}\,,
\]
so that $s^\prime<\frac N {\bar p}$ and $s(\bar p -2)< \bar p^\ast. $ Then,   using Hölder inequality and  \eqref{forte un}  it follows that
\begin{equation*}\label{Vitalin3}
\int_{\Om'}  \frac {b_i(x)^{p'_i} |u_n-u|^{\bar p} }{1+|u_n-u|^2}\le \int_{\Om'}  b_i(x)^{p'_i}   |u_n-u|^{(\bar p-2)} dx\le  \|b_i\|^{p'_i}_{L^{s' \, p'_i}(\Om')}\|u_n-u\|^{(\bar p-2)}_{s(\bar p -2)}\le C  \|b_i\|^{p'_i}_{L^{s' \, p'_i}(\Om')},
\end{equation*}
where $C$ is a constant independent of $n$.  So, also in this case \eqref{compattezza} follows by  Vitali convergence theorem.
So we obtain that for $i=1,\cdots,N$,
\begin{equation}\label{Ip lemmabis}
\int_\Omega \left(\mathcal{A}_i(x,\nabla u_n)-\mathcal{A}_i(x,\nabla u)\right) \left[\partial_{x_i} \eta (u_n-u)\right] \,dx \rightarrow0.
\end{equation}
Now, since  $\eta'(u_n-u)\rightarrow 1$ a.e. in $\Omega$, arguing as in Lemma \ref{conv q.o.} we have that \eqref{Ip lemmabis} implies \eqref{conv qo grad}  for $i=1,\cdots,N$.
 The end of proof runs as in the previous step, obtaining that $u$ is a solution of problem \eqref{P}

\begin{remark} We observe that in case  $p_{\max}:=\max_i p_i >\bar{p}^*$, the compactness of sequence $\mathcal{B}_i(x,u_n)$ in $L^{p'_i}(\Omega)$ holds and then Step 3 in the proof of previous Theorem \ref{TH E1} can be simplified. For completeness we gives some details. In order to obtain  that $\de_{x_i}u_n\to \de_{x_i} u$ a.e. in $\Om$ for every $i=1,...,N$,  we can choose as a test function $\varphi_n=(u_n-u)$ instead of $\varphi_n=\eta (u_n-u)$ in \eqref{ecco} and prove that, unless to pass to a subsequence:
\begin{equation}\label{compattezza2}
  \mathcal \vartheta^i_n(x)  \mathcal B_i(x,u_n) \rightarrow
 \mathcal B_i(x,u)  \,  \text{ strongly in }L^{p'_i}(\Omega) \quad \forall i=1,\cdots,N\,.
\end{equation}
 To prove \eqref{compattezza2},   we note that
 \begin{equation*}\label{Vitalin12}
 \mathcal \vartheta^i_n(x)  \mathcal B_i(x,u_n)   \rightarrow
 \mathcal B_i(x,u)  \qquad \text{a.e.\ in }\Omega\,.
\end{equation*}
and that, using \eqref{b}
for every measurable set $\Om'\subset \Om$ it holds
\begin{equation*}\label{Vitalin2bis}
\int_{\Om'}  \left[  \mathcal \vartheta^i_n(x)  \mathcal B_i(x,u_n)  \right]^{ p'_i}dx\le C\int_{\Om'}
\left[ b_i(x)^{p'_i} |u|^{\bar p} +b_i(x)^{p'_i} |u_n-u|^{\bar p} \right]dx
\end{equation*}
for some positive constant $C=C(\bar p)$. Hence, since $p_{\max} >\bar{p}^*$      and $1<\bar p<N $,   we can choose $s$ satisfying
\[
\frac {N }{N-\bar p} < s \leq  \frac {p_{\max}}{\bar p}\,,
\]
so that $s^\prime<\frac N {\bar p}$ and $s\bar p\leq  p_{\max}. $ Then,   using again Hölder inequality, by Poincaré inequality  it follows that
\begin{equation*}\label{Vitalin3}
\ \int_{\Om'}  b_i(x)^{p'_i}   |u_n-u|^{\bar p} dx\le  \|b_i\|^{p'_i}_{L^{s' \, p'_i}(\Om')}\|u_n-u\|^{\bar p}_{s\bar p }\le C  \|b_i\|^{p'_i}_{L^{s' \, p'_i}(\Om')},
\end{equation*}
where $C$ is a constant independent of $n$. Hence, by  Vitali convergence theorem  \eqref{compattezza2} is proved.
\end{remark}

\begin{remark}\label{ex1}
When $p_i=p<N$ for all $i$ and $N\geq 2$, the compactness of operator
$$T:u\in W_0^{1,\overrightarrow{p}}(\Omega)\rightarrow b|u|^{p-2}u\in L^{p'}(\Omega)$$
could fail. In this case the usual norm in $W_0^{1,p}(\Omega)$ gives an equivalent norm in $W_0^{1,\overrightarrow{p}}(\Omega)$, then  Example 3 in \cite{FGMZ} allows us  to built a sequence of functions $\{u_n\}_{n\in\mathbb{N}}$ in $W_0^{1,\overrightarrow{p}}(\Omega)$ and a function $b\in L^{\frac{N}{p-1},\infty}(\Omega)$ such that $\{\nabla u_n\}_{n\in\mathbb{N}}$ is bounded in $L^{p}(\Omega,\mathbb{R}^N)$, but it is not possible to extract from  $\{b|u_n|^{p-1}\}_{n\in\mathbb{N}}$ any sequence strongly converging in $L^{p'}(\Omega)$.
\end{remark}

\section{Positivity of solutions}

In the isotropic case it is well known (see \cite{B2012}) that when the data is non-negative and the coefficient $b(x)\in L^{N/p-1}(\Omega)$, then the solution is non-negative. In this section we prove a similar result for the anisotropic problem \eqref{P}.
In order to give a precise statement we introduce the following notation: $\mathcal{F}\geq 0$ means $<\mathcal{F},\phi>\geq0$  for every $\phi\geq0$ when $\mathcal{F}\in (W_0^{1,\overrightarrow{p}}(\Omega))^*$. We emphasize that the following result holds  assuming $\mathcal G\equiv 0 $ as well.

\begin{proposition}\label{Prop pos}
Let $u\in W_0^{1,\overrightarrow{p}}(\Omega)$ be a weak solution to problem (\ref{P}) under assumptions ($\mathcal{H}1$)-($\mathcal{H}4$). If $\mathcal{F}\geq 0$, then we have $u\geq0$.
\end{proposition}

\begin{proof}
Let us denote $u^-(x):=\min\{0,u(x)\}$. Taking $T_h(u^-)$ as test function with $h>0$ in \eqref{sol1}, by \eqref{ii2}, \eqref{b}, \eqref{sign} and assumption on the data  we get
\begin{align*}
 \alpha \sum_{i=1}^N \int_{\{-h<u<0\}}|\partial_{x_i}T_h(u^-)|^{p_i} dx &\leq
\sum_{i=1}^N\int_{\{-h<u<0\}} |b_i(x)||u|^{\frac{\bar p}{p'_i}}|\partial_{x_i}T_h(u^-)| dx.
\end{align*}

By Young inequality it follows that
\begin{equation}\label{pos}
 \sum_{i=1}^N \int_{\{-h<u<0\}}|\partial_{x_i}T_h(u^-)|^{p_i} dx \leq
C\sum_{i=1}^N\int_{\{-h<u<0\}} |b_i(x)|^{p'_i}|u|^{\bar p} dx,
\end{equation}
where $C$ is a suitable positive constant. Using the definition \eqref{2.1} in \eqref{pos} we get for every $j$
\begin{align*}
\int_{\{-h<u<0\}}|\partial_{x_j}T_h(u^-)|^{p_j}&\leq
\sum_{i=1}^N\int_{\{-h<u<0\}}|\partial_{x_i}T_h(u^-)|^{p_i}\\
&\leq
C h^{\bar p} \sum_{i=1}^N \|b_i\|^{p_i'}_{ L^{\frac{Np_i'}{\bar p},\infty}(\Omega)}
\int_0^{|\{-h<u<0\}|} s^{- \bar p/N}\, ds.
\end{align*}
By Sobolev inequality \eqref{Sob}, the previous inequality becomes
 \begin{align}\label{pos_2}
 &\left(\int_{\Omega}|T_h(u^-)|^{\bar p^*}\, dx\right)^{\frac{1}{\bar p^*}}
 \leq C
 \prod_{j=1}^N
  \left(
 \int_{\{-h<u<0\}}|\partial_{x_j}T_h(u^-)|^{p_j}\, dx\right)^{\frac{1}{Np_j}}\\
 \nonumber &\leq
  h \left(\int_0^{|\{-h<u<0\}|} s^{- \bar p/N}\, ds\right)^{1/\bar p}
 \prod_{j=1}^N \left(\sum_{i=1}^N \|b_i\|^{p_i'}_{L^{\frac{Np_i'}{\bar p},\infty}(\Omega)}\right)^{\frac{1}{Np_j}}.
 \end{align}
 Now we observe that for $\delta>h$,
 \begin{equation}\label{pos_3}
h|\{u<-\delta\}|^{1/\bar p^*}=\left(\int_{\{u<-\delta\}}|T_h(u^-)|^{\bar p^*}\, dx\right)^{\frac{1}{\bar p^*}}.
 \end{equation}
 Combining \eqref{pos_2} and \eqref{pos_3}, it follows that
 \begin{equation*}
 |\{u<-\delta\}|^{1/\bar p^*}\leq C   \left(\int_0^{|\{-h<u<0\}|} s^{- \bar p/N}\, ds\right)^{1/\bar p}
 \prod_{j=1}^N \left(\sum_{i=1}^N \|b_i\|^{p_i'}_{L^{\frac{Np_i'}{\bar p},\infty}(\Omega)}\right)^{\frac{1}{Np_j}}.
\end{equation*}
Observing that $$\int_0^{|\{-h<u<0\}|} s^{- \bar p/N}\, ds\rightarrow 0 \text{ as } h\rightarrow0^+,$$
we conclude that  the measure of the set $\{u<-\delta\}$ is zero for every $\delta>0$.

\medskip
\end{proof}

\bigskip
As a consequence of Proposition \ref{Prop pos} we get the following partial uniqueness result when $\mathcal{F}\equiv0$.

\begin{corollary}\label{Cor unic}
Let $u\in W_0^{1,\overrightarrow{p}}(\Omega)$ be a weak solution to problem (\ref{P}) under assumptions ($\mathcal{H}1$)-($\mathcal{H}4$). If $\mathcal{F}\equiv0$, then  we have $u\equiv0$.
\end{corollary}

\begin{proof}
The thesis follows arguing as in the previous proposition taking $T_h(u^-)$ and $T_h(u^+)$ as test function.
\end{proof}

\bigskip

We stress that a similar argument can be used to prove the uniqueness when $p_i=2$ for all $i$ and the coefficient $b_i(x)=b(x)\in L^{N,\infty}(\Omega)$ (see \cite{B2012} when $b_i(x)=b(x)\in L^N(\Omega)$).

\section{Uniqueness results}\label{sez unic}

 First of all  we observe that prototype \eqref{model 1} verifies a strongly monotone condition, but not
the assumption of Lipschitz continuity in $u$ when  $p_i$ are not all equals and  $p_i\leq2$ for $i=1,\cdots,N$ as required in the classical uniqueness result.
%(see \cite{DF} for anisotropic case).
Then the novelty of our result consists in the possibility to
deal with cases when $\mathcal{B}(x,u)$ is only H\"{o}lder continuous with respect to $u$ (see \eqref{cond B} below). In general the presence of a zero order term could help in order to obtain uniqueness result, even in the case $p_i\geq2$ for $i=1,\cdots,N$. In what follows we assume that
\begin{equation}\label{cond G}%
\mathcal{G}\left(x,s\right)  \text{ is strictly monotone increasing function in $s$.}
\end{equation}

Note that whenever $\mathcal G$ satisfies \eqref{cond G} and $\mathcal G (x,0)\equiv 0$ for all $x\in\Omega,$ then \eqref{sign} obviously holds. For example $\mathcal{G}\left( x,s\right) =\widetilde{\mu}|s|^{\gamma-1}s$ with $\widetilde{\mu}>0,0<\gamma< p_\infty-1$  verifies \eqref{cond G} and \eqref{b1} and \eqref{sign}.
 We observe that we can relax condition \eqref{cond G} assuming that
\begin{equation*}
 \left(\mathcal{G}\left(  x,s\right)-\mathcal{G}\left(  x,s'\right)\right)(s-s')>0 \text{ for } s> s'.\label{cond G2}%
\end{equation*}

As usual when we deal with uniqueness results for equation in which we have the dependence on the unknown $u$ we consider two cases: first when $p_i\leq2$ for every $i=1,\cdots,N$, then when $p_i\geq2$ for every $i=1,\cdots,N$ and finally the mixed case.

\subsection{First case: $p_i\leq2$ for every $i=1,\cdots,N$}
In this subsection we prove the uniqueness of weak solutions of problem \eqref{P} when all $p_i\leq 2$.
For example in model cases \eqref{model 1} the main difficulty is due to the terms $B_i(x,u)=\beta_i(x)|u|^{\frac{\bar p}{p'_i}-1}u$ with $\beta_i\in L^{\frac{N p'_i}{\bar p},\infty}(\Omega)$, who are H\"{o}lder continuous but not Lipschitz continuous with respect to solution $u$. To prove uniqueness we will strongly use the presence of zero order term.
We assume that
\begin{equation}
\left(  \mathcal{\mathcal{A}}_{i}\left(  x,\xi\right)  -\mathcal{A}_{i}\left(  x,\xi^{\prime}\right)
\right)  \left(  \xi_{i}-\xi_{i}^{\prime}\right)  \geq\widetilde{\alpha}\left(\left\vert \xi_{i}\right\vert +\left\vert \xi_{i}^{\prime
}\right\vert \right) ^{p_{i}-2}\left\vert \xi_{i}-\xi_{i}^{\prime}\right\vert
^{2} \label{monotonia forte}%
\end{equation}
with $\widetilde{\alpha}>0$ and
\begin{equation}\label{cond B}
\left\vert \mathcal{B}_{i}\left(  x,s\right)  -\mathcal{B}_{i}\left(  x,s^{\prime}\right)
\right\vert \leq \widetilde{b}_i(x)
\left\vert s-s^{\prime}\right\vert^{\frac{\bar p}{p_i'}}
\end{equation}
with $\widetilde{b}_i: \,\Om\to [0,+\infty)$ measurable function such that
\begin{equation}\label{riBIS}
 \widetilde{b}_i\in L^{\frac{N\,p'_i} {\bar p},\infty}(\Om),
\end{equation}
for $i=1,..,N$.
We stress that  $\mathcal{A}_{i}\left(  x,\xi\right)=\alpha |\xi_i|^{p_i-2}\xi_i$
verifies \eqref{monotonia forte} and  $\mathcal{B}_{i}\left(  x,s\right) = \beta_i(x)|s|^{\frac{\bar p}{p_i'}-1}s$  verifies \eqref{cond B}, recalling that $\frac{\bar p}{p_i'}\leq 1$, because $\bar p\leq2\leq p_i'$ .

\medskip

\begin{theorem}\label{THU1}
%Let us assume $\Omega$ has Lipschitz boundary, $p_{i}\leq2$ for $i=1,..,N$,  $\bar{p}<N$, that ($\mathcal{H}1$)-($\mathcal{H}4$) are fulfilled, that $\mathcal{G}\not \equiv 0$ and  there exists a positive constant  $d=d(N, \alpha, \vec{p})$  such that \eqref{hpdistmain} holds.
Let $u\in W_0^{1,\overrightarrow{p}}(\Omega)$ be a weak solution to problem (\ref{P}) under assumptions ($\mathcal{H}1$)-($\mathcal{H}4$)  with $p_{i}\leq2$ for $i=1,..,N$, $\bar p<N$  and $\mathcal{G}\not \equiv 0$. If $2\min_i\{\frac{\bar p}{p_i'}\}\geq1$, (\ref{cond G}),(\ref{monotonia forte}) and (\ref{cond B}) are in force, then $u$ is the unique weak solution to problem \eqref{P}.
\end{theorem}

\medskip

\begin{proof}
We generalize the ideas contained in   \cite{CC} where a linear isotropic operator is considered.

Let $u$ and $v$ be two weak solutions to problem (\ref{P}). Let us denote
$w=\left(  u-v\right)  ^{+}$ and  $D=\left\{  x\in\Omega:w>0\right\}$. Denoting $r_{\min}=\min_i \{\frac{\bar p}{p_i'}\}$ and recalling that  $2r_{\min}\geq1$, for every $\varepsilon>0$ there exists $\delta(\varepsilon)$ such that
$$\int_{ \delta(\varepsilon)}^\varepsilon\frac{1}{s^{2r_{\min}}}ds=1.$$
Let us define
\begin{equation}\label{Psi}
\Psi^\sigma_\varepsilon(t)=
\left\{
\begin{array}[c]{ll}%
0 &\hbox { se } t\leq\delta(\varepsilon),
\\ \int_{ \delta(\varepsilon)}^t\frac{1}{s^{\sigma}}ds   & \hbox{ se } \delta(\varepsilon)<t< \varepsilon
 \\1  & \hbox{ se } t\geq \varepsilon.
\end{array}
\right.
\end{equation}
We take $\Psi_\varepsilon(t)=\Psi^\sigma_\varepsilon$ with $\sigma=2r_{\min}$. We stress that $\Psi_\varepsilon(t)$ is Lipschitz continuous, $\Psi'_\varepsilon(t)\geq0$ and $\varphi\Psi_\varepsilon(w)\in W_0^{1,\overrightarrow{p}}(\Omega)$ taking $\varphi\geq0$ and $\varphi \in W^{1,\overrightarrow{p}}(\Omega)\cap L^\infty(\Omega)$.

 Supposing that $D$ has positive measure, we use $\varphi\Psi_\varepsilon(w)$ as test function in the difference of the equations. It follows
\begin{align*}
I_\varepsilon:=&
\overset{N}{\underset{i=1}{\sum}} \int_{\Omega}
\left\{
\left[  \mathcal{A}_{i}\left(  x,\nabla
u\right)  -\mathcal{A}_{i}\left(  x,\nabla v\right)  \right] +
\left[  \mathcal{B}_{i}\left(  x,u\right)  -\mathcal{B}_{i}\left(  x, v\right)  \right] \right\} \partial_{x_{i}}%
\varphi \Psi_\varepsilon(w)\,dx\\
&+\int_{\Omega}\left[\mathcal{G}(x,u)-\mathcal{G}(x,v)\right]\varphi \Psi_\varepsilon(w) \,dx=
\\
&-
\overset{N}{\underset{i=1}{\sum}}\int_{\Omega}
\left\{
\left[  \mathcal{A}_{i}\left(  x,\nabla
u\right)  -\mathcal{A}_{i}\left(  x,\nabla v\right)  \right] + \left[  \mathcal{B}_{i}\left(  x,u\right)  -\mathcal{B}_{i}\left(  x, v\right)  \right]\right\}  \partial_{x_{i}}w
\Psi'_\varepsilon(w)\varphi\,dx.
\end{align*}
By (\ref{monotonia forte}) and (\ref{cond B}) we get
\begin{align*}
-\overset{N}{\underset{i=1}{\sum}}
\int_{\Omega}&
\left\{
\left[  \mathcal{A}_{i}\left(  x,\nabla
u\right)  -\mathcal{A}_{i}\left(  x,\nabla v\right)  \right]  \partial_{x_{i}}w
\Psi'_\varepsilon(w)\varphi+
\left[  \mathcal{B}_{i}\left(  x,u\right)  -\mathcal{B}_{i}\left(  x, v\right)  \right]
\right\}
 \partial_{x_{i}}w
\Psi'_\varepsilon(w)\varphi\, dx
\\&
\leq-\overset{N}{\underset{i=1}{\sum}}\int_{\Omega}
\widetilde{\alpha}\frac{|\partial_{x_{i}}w|^2}{(|\partial_{x_{i}}u|+ |\partial_{x_{i}}v|)^{2-p_i}}
\Psi'_\varepsilon(w)\varphi\, dx+\overset{N}{\underset{i=1}{\sum}}\int_{\Omega}
\widetilde{b}_i(x) w^{r_i} | \partial_{x_{i}}w|
\Psi'_\varepsilon(w)\varphi\, dx,
\end{align*}
where $r_i=\frac{\bar p}{p'_i}$.  Using Young inequality with $\theta<\widetilde{\alpha}$  we obtain

\begin{align*}
-\overset{N}{\underset{i=1}{\sum}}
\int_{\Omega}&
\left\{
\left[  \mathcal{A}_{i}\left(  x,\nabla
u\right)  -\mathcal{A}_{i}\left(  x,\nabla v\right)  \right] +
\left[  \mathcal{B}_{i}\left(  x,u\right)  -\mathcal{B}_{i}\left(  x, v\right)  \right]\right\}  \partial_{x_{i}}w
\Psi'_\varepsilon(w)\varphi\, dx
\\
\leq&-\overset{N}{\underset{i=1}{\sum}}\int_{\Omega}
(\widetilde{\alpha}-\theta)
\frac{|\partial_{x_{i}}w|^2}{(|\partial_{x_{i}}u|+ |\partial_{x_{i}}v|)^{2-p_i}}
\Psi'_\varepsilon(w)\varphi \, dx+\\
& +C(\theta)\overset{N}{\underset{i=1}{\sum}}\int_{\Omega}
\widetilde{b}^2_i(x) w^{2r_i} (|\partial_{x_{i}}u|+ |\partial_{x_{i}}v|)^{2-p_i}
\Psi'_\varepsilon(w)\varphi\, dx
\\
&\leq C(\theta)\|\varphi\|_\infty\overset{N}{\underset{i=1}{\sum}}\varepsilon^{2(r_i-r_{\min})}
\int_{\{\delta(\varepsilon)<w<\varepsilon\}}
\widetilde{b}^2_i(x)(|\partial_{x_{i}}u|+ |\partial_{x_{i}}v|)^{2-p_i}\, dx= J_\varepsilon.
\end{align*}
We underline that $\widetilde{b}^2_i(x)(|\partial_{x_{i}}u|+ |\partial_{x_{i}}v|)^{2-p_i}\in L^1(\Omega)$. Indeed we have
\begin{equation*}
\begin{split}
\int_{\Omega}\widetilde{b}^2_i(x)(|\partial_{x_{i}}u|+ |\partial_{x_{i}}v|)^{2-p_i}\, dx&\leq
\left(\int_{\Omega}\widetilde{b}_i^{p'_i}(x) \,dx\right)^{\frac{2}{p'_i}}
\left(\int_{\Omega} (|\partial_{x_{i}}u|+ |\partial_{x_{i}}v|)^{p_i} \,dx\right)^{\frac{2-p_i}{p_i}}\\
&\leq C(\Omega,p_i)\|\widetilde{b}_i\|^{p'_i}_{ L^{\frac{N\,p'_i} {\bar p},\infty}(\Om)}
\left(\int_{\Omega} (|\partial_{x_{i}}u|+ |\partial_{x_{i}}v|)^{p_i} \,dx\right)^{\frac{2-p_i}{p_i}}.
\end{split}
\end{equation*}
%and $$\int_{\Omega} (|\partial_{x_{i}}u|+ |\partial_{x_{i}}v|)^{p_i} \,dx\leq C(\Omega,p_i)\|b_i\|_{ L^{\frac{N\,p'_i} {\bar p},\infty}(\Om)}.$$
Since $(r_i-r_{\min})\geq0$, letting $\varepsilon\rightarrow0$ we get $J_\varepsilon\rightarrow0$. Moreover $\Psi_\varepsilon(w)\rightarrow \chi_{\{u-v>0\}}$, where $\chi_{\{u-v>0\}}$ denotes the characteristic function of the set $\{u-v>0\}$. By the Lebesgue dominated convergence Theorem we obtain
\begin{equation}\label{E1}
\overset{N}{\underset{i=1}{\sum}}\int_{\{u-v>0\}}
\left\{
\left[  \mathcal{A}_{i}\left(  x,\nabla
u\right)  -\mathcal{A}_{i}\left(  x,\nabla v\right)  \right] +
\left[  \mathcal{B}_{i}\left(  x,u\right)  -\mathcal{B}_{i}\left(  x, v\right)  \right]\right\}  \partial_{x_{i}}%
\varphi
+\left[\mathcal{G}(x,u)-\mathcal{G}(x,v)\right]\varphi \,dx\leq0.
\end{equation}
Taking $\varphi=1$ in \eqref{E1} it follows
$$\int_{\{u-v>0\}}\left[\mathcal{G}(x,u)-\mathcal{G}(x,v)\right] \,dx\leq0.$$
Assumption \eqref{cond G} allows us to conclude that ${\{u-v>0\}}$ ha zero measure. Exchanging $u$ for $v$ we conclude.
\end{proof}

\medskip

\begin{remark}
\begin{itemize}
\item [i)]  If $3/2\leq p_i\leq2$ for all $i$ then $2\min_i\{\frac{\bar p}{p'_i}\}\geq1$.

\item [ii)] Obviously when all $p_i=2$ the uniqueness results holds even $\mathcal{G}\equiv0$, because we have Lipschitz dependence on $u$.

\item [iii)] Theorem \ref{THU1} holds replacing the condition $p_i\leq2$ for every $i=1,\cdots,N$ with $\frac{\bar p}{p'_i}\leq 1$ for every $i=1,\cdots,N$.
\end{itemize}
\end{remark}

\subsection{Second case: $p_i\geq2$ for every $i=1,\cdots,N$}
In this subsection we prove the uniqueness of weak solutions of problem \eqref{P} when all $p_i\geq 2$.
In the model cases \eqref{model 1} the term $B_i(x,u)=\beta_i(x)|u|^{\frac{\bar p}{p'_i}-1}u$ with $|\beta_i|\in L^{\frac{N p'_i}{\bar p},\infty}(\Omega)$  are locally Lipschitz continuous with respect to solution $u$, but it is well-known that even in the isotropic case when $p>2$ the uniqueness can fail (see for example in \cite{ABM} at the end of Section 2). Again to prove uniqueness we will strongly use the presence of zero order term.
We assume
\begin{equation}
\left(  \mathcal{\mathcal{A}}_{i}\left(  x,\xi\right)  -\mathcal{A}_{i}\left(  x,\xi^{\prime}\right)
\right)  \left(  \xi_{i}-\xi_{i}^{\prime}\right)  \geq\widehat{\alpha}\left\vert \xi_{i}-\xi_{i}^{\prime}\right\vert
^{p_{i}} \label{monotonia forte_bis}%
\end{equation}
with $\widehat{\alpha}>0$ and
\begin{equation}\label{cond B_bis}
\left\vert \mathcal{B}_{i}\left(  x,s\right)  -\mathcal{B}_{i}\left(  x,s^{\prime}\right)
\right\vert \leq \widehat{b}_i(x)
\left\vert s-s^{\prime}\right\vert
\left(|s|+|s^\prime|+\zeta\right)^{\frac{\bar p}{p'_i}-1}
\end{equation}
with $ \widehat{b}_i(x)$ defined as in \eqref{riBIS} and $\zeta\geq0$.
We stress that $\frac{\bar p}{p'_i}\geq1$ under our assumptions. Moreover $\mathcal{A}_{i}\left(  x,\xi\right)=\alpha |\xi_i|^{p_i-2}\xi_i$ verifies \eqref{monotonia forte_bis} and $\mathcal{B}_{i}\left(  x,s\right) =\beta_i(x)|s|^{\frac{\bar p}{p_i'}-1}s$  verifies \eqref{cond B}, recalling that $p_i'\leq2\leq\bar p$.

\medskip

\begin{theorem}\label{THU2}
Let $u\in W_0^{1,\overrightarrow{p}}(\Omega)$ be a weak solution to problem (\ref{P}) under assumptions ($\mathcal{H}1$)-($\mathcal{H}4$)  with $p_{i}\geq2$ for $i=1,..,N$, $\bar p<N$ and $\mathcal{G}\not \equiv 0$.  If (\ref{cond G}), (\ref{monotonia forte_bis}) and (\ref{cond B_bis}) are in force, then $u$ is the unique weak solution to problem \eqref{P}.
\end{theorem}

\medskip

\begin{proof}
The idea of the proof follows the previous theorem. Let $u$ and $v$ be two weak solutions to problem (\ref{P}). Let us denote
$w=\left(  u-v\right)  ^{+}$ and  $D=\left\{  x\in\Omega:w>0\right\}$.

Let $\Psi_\varepsilon(t)=\Psi^\sigma_\varepsilon(t)$  with $\sigma=\min_i {p'_i}$ defined in \eqref{Psi}.
Supposing that $D$ has positive measure, we use $\varphi\Psi_\varepsilon(w)$ as test function in the difference of the equations. By
(\ref{monotonia forte_bis}) and (\ref{cond B_bis}) we get
\begin{align*}
-\overset{N}{\underset{i=1}{\sum}}
\int_{\Omega}&
\left\{
\left[  \mathcal{A}_{i}\left(  x,\nabla
u\right)  -\mathcal{A}_{i}\left(  x,\nabla v\right)  \right]  \partial_{x_{i}}w
\Psi'_\varepsilon(w)\varphi+
\left[  \mathcal{B}_{i}\left(  x,u\right)  -\mathcal{B}_{i}\left(  x, v\right)  \right]
\right\}
 \partial_{x_{i}}w
\Psi'_\varepsilon(w)\varphi\, dx
\\&
\leq-\overset{N}{\underset{i=1}{\sum}}\int_{\Omega}
\widehat{\alpha}|\partial_{x_{i}}w|^{p_i}
\Psi'_\varepsilon(w)\varphi\, dx+\overset{N}{\underset{i=1}{\sum}}\int_{\Omega}
\widehat{b}_i w \left(|u|+|v|+\zeta\right)^{r_i-1} | \partial_{x_{i}}w|
\Psi'_\varepsilon(w)\varphi\, dx,
\end{align*}
where we put $r_i=\frac{\bar p}{p'_i}$.
Using Young inequality with $\theta<\widehat{\alpha}$ we obtain
\begin{align*}
-\overset{N}{\underset{i=1}{\sum}}
\int_{\Omega}&
\left\{
\left[  \mathcal{A}_{i}\left(  x,\nabla
u\right)  -\mathcal{A}_{i}\left(  x,\nabla v\right)  \right] +
\left[  \mathcal{B}_{i}\left(  x,u\right)  -\mathcal{B}_{i}\left(  x, v\right)  \right]\right\}  \partial_{x_{i}}w
\Psi'_\varepsilon(w)\varphi\, dx
\\
\leq&-\overset{N}{\underset{i=1}{\sum}}\int_{\Omega}
(\widehat{\alpha}-\theta)
|\partial_{x_{i}}w|^{p_i}
\Psi'_\varepsilon(w)\varphi \, dx+\\
& +C(\theta)\overset{N}{\underset{i=1}{\sum}}\int_{\Omega}
\widehat{b}_i^{p'_i} w^{p'_i} \left(|u|+|v|+\zeta\right)^{(r_i-1)p'_i}
\Psi'_\varepsilon(w)\varphi\, dx
\\
&\leq C(\theta)\|\varphi\|_\infty\overset{N}{\underset{i=1}{\sum}}\varepsilon^{(p'_i-\min_i p'_i)}
\int_{\{\delta(\varepsilon)<w<\varepsilon\}}
\widehat{b}_i^{p'_i}\left(|u|+|v|+\zeta\right)^{(\bar p-p'_i )}\, dx= J_\varepsilon.
\end{align*}
We underline that $\widehat{b}^{p'_i}_i \left(|u|+|v|+\zeta\right)^{(\bar p-p'_i )}\in L^1(\Omega)$. Indeed  we have
\begin{equation*}
\begin{split}
\int_{\Omega}
\widehat{b}_i^{p'_i}\left(|u|+|v|+\zeta\right)^{(r_i-1)p'_i}\, dx
&\leq C\left(\|\widehat{b}_i\|^{^{p'_i}}_{ L^{\frac{N\,p'_i} {\bar p},\infty}(\Om)}
\||u|+|v|\|^{\bar p-p'_i}_{L^{\frac{N(\bar p-p'_i)}{N-\bar p},\bar p-p'_i}(\Omega)}
+  \|\widehat{b}_i\|^{p'_i}_{ L^{p'_i} (\Om)}\right).
\end{split}
\end{equation*}
The right-hand side is finite, since $u,v\in L^{\frac{N(\bar p-p'_i)}{N-\bar p},\bar p-p'_i}(\Omega)$ and $L^{\bar p^*,\bar p}(\Omega)\subset L^{\frac{N(\bar p-p'_i)}{N-\bar p},\bar p-p'_i}(\Omega)$. We conclude as in Theorem \ref{THU1} letting $\varepsilon\rightarrow0$.
\end{proof}

\medskip

\begin{remark}
Theorem \ref{THU1} holds replacing the condition $p_i\geq2$ for every $i=1,\cdots,N$ with $\frac{\bar p}{p'_i}\geq 1$ for every $i=1,\cdots,N$.
\end{remark}

\subsection{The mixed case:  when $\min_i p_i\leq2$ and $\max_i p_i>2$.}
We end the section studying the mixed case when $\min_i p_i\leq2$ and $\max_i p_i>2$. For simplicity of presentation we  take into account the model problem \eqref{model 1} when $\min_i p_i\leq2$ and $\max_i p_i\geq2$.

\begin{theorem}\label{THU3}
Let $u\in W_0^{1,\overrightarrow{p}}(\Omega)$ be a weak solution to problem \eqref{model 1}  with $p_{i}>1$ for $i=1,..,N$, $\min_i p_i\leq2$ and $\max_i p_i>2$,$\bar{p}<N$, $\widetilde{\mu}>0$ and $\beta_i\in L^{\frac{Np'_i}{\bar p},\infty}(\Omega)$. If $2\min_i\{\frac{\bar p}{p'_i}\}\geq1$, then $u$ is the unique weak solution to problem \eqref{model 1}.
\end{theorem}

We stress that problem \eqref{model 1} verifies assumptions (\ref{monotonia forte}) and (\ref{cond B}) for $p_i\leq2$. Otherwise for $p_i>2$ conditions (\ref{monotonia forte_bis}) and (\ref{cond B_bis}) are fulfilled. Then we can blend the proof of Theorem \ref{THU1} and Theorem \ref{THU2} to prove Theorem \ref{THU3}. In particular $\Psi_\varepsilon(t)=\Psi^\sigma_\varepsilon(t)$  with $\sigma=\min_i\{p'_i,2\frac{\bar p}{p'_i}\}$ defined in \eqref{Psi}.

\section*{Acknowledgments}
The authors are partially supported by GNAMPA of the Italian INdAM (National Institute of High Mathematics).

\end{document}